\documentclass{amsart}
\usepackage{amssymb}
\usepackage{multirow}
\usepackage{xy}
\usepackage{textcomp}      
\usepackage{enumitem}
\usepackage{tikz}
\usepackage{oplotsymbl} 

\xyoption{all}

\newtheorem{thm}{Theorem}[section] 
\newtheorem*{thm*}{Theorem} 
\newtheorem{algo}[thm]{Algorithm}

\newtheorem{cor}[thm]{Corollary}
\newtheorem{defn}[thm]{Definition}

\newtheorem{exmpl}[thm]{Example}

\newtheorem{prop}[thm]{Proposition}

\newtheorem{rem}[thm]{Remark}

\newtheorem{ques}[thm]{Question}

\newcommand\Lection[6][]{{\newpage 
\section[#2]{\underline{{{\sc{#2}}\if!#3!{}\else{ (#3)}\fi}}}
\if!#1!{}\else{\tiny{\begin{center}(#1)\end{center}}}\fi\flushright{{\tiny{#5}}} 
\begin{center}{{\bf{#4}}}{\ } \\{#6}\end{center}\vskip1cm}}

\def\eg{{e.g.}}

\def\lam{{\lambda}}

\def\s{\sigma}

\def\Z {{\mathbb {Z}}}

\def\EE{{\mathsf E}}

\def\GG{{\mathcal G}}

\DeclareMathAlphabet{\mathscr}{OT1}{pzc}{m}{it}

\def\ra{{\rightarrow}}
\def\lra{{\,\longrightarrow\,}}

\def\hra{{\;\hookrightarrow\;}} 

\def\sub{\subseteq}
\def\({\left(}
\def\){\right)}
\def\isom{{\;\cong\;}} 

\def\semidirect{{\ltimes}} 
\def\co{{\,{:}\,}}
\def\divides{{\,|\,}}           

\newcommand\suchthat{{\,:\ \,}}
\newcommand\subjectto{{\,|\ }}

\newcommand\comp[1]{{#1^{\operatorname{c}}}} 

\newcommand{\opname}[1]{{\operatorname{#1}}}
\newcommand{\oper}[1]{{\operatorname{#1}}}
\newcommand\operA[2]{{\if!#2!\operatorname{#1}\else{\operatorname{#1}_{#2}^{\phantom{I}}}\fi}} 

\newcommand{\Tr}[1][]{\if!#1!\operatorname{Tr}\else{\operatorname{Tr}_{#1}^{\phantom{I}}}\fi} 
\newcommand{\sgn}[2][!]{{\operatorname{sgn}\if!#1(#2)\else\relax\fi}}
\DeclareMathOperator{\rank}{{{rank}}} 

\DeclareMathOperator{\Aut}{Aut} %
\newcommand\algint[2][]{\if!#1\relax O_{#2}\else{O_{#2}^{\phantom{I}}}\fi}

\renewcommand\H[4][!]{{\operatorname{H}^{#2}\!\!\;({#3},{#4}{\if!#1\relax\else(#1)\fi})}}

\newcommand\cond[2][!]{{\operatorname{cond}_{\if!#1\relax\else{\comp{#1}}\fi}(#2)}}

\newcommand\abs[2][]{\left|{#2}\right|_{#1}} 

\newcommand{\set}[1]{{\left\{#1\right\}}}

\newcommand{\card}[1]{{\left|{#1}\right|}}
\newcommand\ideal[1]{{\left<{#1}\right>}}
\newcommand\sg[1]{{\ideal{#1}}}

\newcommand\dist[3][]{{\oper{dist}_{#1}(#2,#3)}} 

\newcommand\Pic[1]{{\operatorname{Pic}}} 
\newcommand\Div[1]{{\operatorname{Div}}} 

\newcommand\Cref[1]{{Corollary~\ref{#1}}}
\newcommand\Dref[1]{{Definition~\ref{#1}}}
\newcommand\Eref[1]{{Example~\ref{#1}}}
\newcommand\Erefs[2]{{Examples~\ref{#1} and~\ref{#2}}}

\newcommand\Pref[1]{{Proposition~\ref{#1}}}
\newcommand\Rref[1]{{Remark~\ref{#1}}}
\newcommand\Tref[1]{{Theorem~\ref{#1}}}
\newcommand\Fref[1]{{Figure~\ref{#1}}}
\newcommand\Sref[1]{{Section~\ref{#1}}}
\newcommand\Srefs[2]{{Sections~\ref{#1} and~\ref{#2}}}
\newcommand\Ssref[1]{{Subsection~\ref{#1}}}

\newcommand\eq[1]{{(\ref{#1})}}

\newcommand\Eq[1]{{Equation \eq{#1}}}

\long\def\half#1\halved{{\footnotesize{#1}}}

\newcommand\thisfile{{\jobname.tex}}
\newcommand\paper[7]{{{#1},\ {\it{#2}},\ {#3}\ {\bf{#4}}\if!#5!\else(#5)\fi,\ {#6},\ ({#7}).}} 
\newcommand\book[4]{{{#1},\ {{#2}}{\if!#3!\relax\else{,\ {#3}}\fi}{\if!#4!\relax\else{,\ {#4}}\fi}.}} 
\newcommand\submitted[3]{{{#1},\ {\it{#2}}, submitted{\if!#3!{}\else, ({#3})\fi}.}} 

\newcommand\dd[4][!]{\abs[#2]{\:\!\det{\if!#1\relax\:\!\!\else(#1)\fi}}^{#3} #4}

\newcommand\CAT{{\mathfrak{C}}}
\newcommand\Aref[1]{{Algorithm~\ref{#1}}}

\newcommand\BIT{\begin{enumerate}}
\newcommand\EIT{\end{enumerate}}

\newcommand\Y{{\mbox{$\mathsf{Y}$}}}
\newcommand\CY{{C_{\mathsf{Y}}}}
\newcommand\AY{{A_{\mathsf{Y}}}}
\newcommand\path[1][n]{{\mathfrak{p}_{#1}}}
\newcommand\cyc[1][n]{{\mathfrak{c}_{#1}}}
\newcommand\Line[1]{{\mathcal{L}{#1}}}
\newcommand\Cover[1][\Sigma]{{W(\covg{#1})}}
\newcommand\covg[1]{{\mathcal{L}^{[2]}#1}}
\newcommand\ab[1]{{#1/[#1,#1]}}
\newcommand\HNN{{\opname{HNN}}}
\newcommand\cycpath[2]{{\cyc[#1]\vee\path[#2]}}

\normalfont

\title{Stabilizers in Coxeter groups with a transpositional action}

\def\UVemail{vishne@math.biu.ac.il}

\author[N.~Cohen]{Noa Cohen}
\author[U.~Vishne]{Uzi Vishne}
\email{noa.cohen.1@gmail.com}
\email{\UVemail}
\address{Department of Mathematics, Bar-Ilan University}

\date\today
\thanks{Partially supported by Israel Science Foundation grant 1994/20.}


\begin{document}

\begin{abstract}
We study simply-laced Coxeter groups whose Coxeter diagram is the line graph $\Line{\Gamma}$ of some simple graph $\Gamma$; such groups act on the set of vertices of the graph by transpositions. Our goal is to describe the point stabilizer of this action in detail. For any tree $\Sigma$, we show that the stabilizer is a finite-index reflection subgroup, and identify its Coxeter diagram, thereby discovering finite-index embeddings of Coxeter groups $\Cover \hra W(\Line{\Sigma})$ for any tree $\Sigma$, including infinitely many cases where the stabilizer itself is simply-laced.
\end{abstract}


\maketitle

\def\CoxeterDiagI{{*={\bullet} \ar@{-}[rr] \ar@{-}[rd] & {} & *={\bullet} \ar@{-}[ld] \ar@{-}[rr] \ar@{-}[rd] & {} & *={\bullet}  \ar@{-}[ld]\\
{} & *={\bullet} \ar@{-}[rr]  & {} & *={\bullet}  & {}}}
\def\CoxeterDiagII{{*={\bullet} \ar@{-}[r] & *={\bullet} \ar@{-}[rd] \ar@{-}[rr] & {} & *={\bullet} \ar@{-}[ld] & *={\bullet} \ar@{-}[l] \\
{} & {} & *={\bullet}  & {} &  {}}}
\def\CoxeterDiagIII{{*={\bullet} \ar@{-}[r] & *={\bullet} \ar@{-}[r] \ar@{-}[d]& *={\bullet} \\
*={\bullet} \ar@{-}[r] & *={\bullet} \ar@{-}[r] & *={\bullet}}}

\section{Introduction}

The structure of Coxeter groups is quite well understood when the group acts simplicially on spherical, affine or hyperbolic spaces. However, the vast majority of Coxeter groups do not have such an action, and their inner structure is more challenging. For example, there is no systematic study of commensurability of Coxeter groups (commensurability was studied in \cite{Dotti} for hyperbolic Coxeter groups, and in \cite{DST} for right-angled Coxeter groups). Deodhar \cite{Deodhar} and Dyer \cite{Dyer} proved that a subgroup of a Coxeter group, generated by reflections, is a Coxeter group. Finite-index reflection subgroups were studied in \cite{Felikson1}, while \cite{Felikson2} gives a criterion in the odd-angled case.

In this paper we study simply-laced Coxeter groups whose Coxeter diagrams are line graphs (see \Ssref{ssec:lg}). Such groups are endowed with a natural action on a finite set. Our goal is to describe the group structure of the point stabilizer of this action. When the Coxeter diagram is the line graph of a tree, the stabilizer is itself a Coxeter group, providing infinitely many examples of finite-index embeddings of Coxeter groups (which are not odd-angled).
For example, there is an embedding
$$W(\raisebox{2.2ex}{\xymatrix@R=8pt@C=1pt\CoxeterDiagI}) \ \hra \ W(\raisebox{2.2ex}{\xymatrix@R=8pt@C=4pt\CoxeterDiagII}),$$
as a subgroup of index $6$, where $W(\cdot)$ is the Coxeter group associated to a given diagram (take $\Gamma = \EE_6$ in \Tref{nice}).

\medskip

When $\Gamma$ is a connected graph, the transpositions corresponding to the edge set $E(\Gamma)$ generate the full symmetric group. In fact, $\Gamma$ defines a simply-laced Coxeter group $C(\Gamma)$, whose Coxeter diagram is the line graph of~$\Gamma$. What we obtain in this manner is a Coxeter group with a transpositional action on the set of vertices (\Ssref{ssec:ta}). A presentation of the symmetric group on this set of generators was given in \cite{Serg, Sol}.

Let $n$ denote the number of vertices of the graph $\Gamma$.
The map $C(\Gamma) \ra S_n$ splits through a quotient $\CY(\Gamma)$, introduced in~\cite{Cox}. This group was shown there to be a semidirect product $S_n \ltimes F_{t,n}$, where $F_{t,n}$ is a subdirect product of free groups. The associated Artin group, $\AY(\Gamma)$, was defined in~\cite{Artin} when $\Gamma$ is a planar graph, and shown there to be the semidirect product $B_n \ltimes \hat{F}_{t,n}$, where $\hat{F}_{t,n}$ is a finite central extension of~$F_{t,n}$. The ability to solve the word problem in these  groups allows one to compute the fundamental groups of Galois covers of surfaces, see \cite{Merav8} and many references therein. By contrast, $C(\Gamma)$ is almost never a semidirect product, which again makes this group more complicated to handle.

\medskip

\noindent
Denote by $C(\Gamma)_o$ the stabilizer of some vertex $o$. Here are our main results.
\begin{enumerate}
\item We give a presentation for $C(\Gamma)_o$ (\Pref{thepres}). For example, it follows that the abelianization of $C(\Gamma)_o$ is determined by the rank of $\pi_1(\Gamma)$ (\Tref{abt}).
\item For a tree $\Sigma$, we show that $C(\Sigma)_o$ is a reflection subgroup, identified as a Coxeter group (\Tref{main-1}), with a concrete Concrete diagram (\Tref{nice}). We thus obtain an infinite family of finite-index embeddings of Coxeter groups. Furthermore, in infinitely many cases both groups are simply-laced (\Eref{ex69}).
\item When $\Gamma$ is unicyclic, $C(\Gamma)_o$ admits a relative conjugation presentation over a Coxeter group: it is obtained by adjoining one generator and imposing explicit conjugation relations (\Cref{caseuni}). In the special cases where $\Gamma$ is a cycle or a cycle with a tail, $C(\Gamma)_o$ is in fact an HNN extension of a Coxeter group; in the latter case the relevant parabolic subgroups are explicitly presented (\Cref{O--}).
\item In general, $C(\Gamma)_o$ admits an iterated relative conjugation presentation beginning with a Coxeter group (\Tref{thegen}). When $\Gamma$ has a separating spanning subtree, all the additional generators and conjugation relations can be introduced simultaneously (\Tref{thethin}).
\end{enumerate}

\bigskip
The main characters of the show, $C(\Gamma)$ and the stabilizer $C(\Gamma)_o$, are introduced in \Sref{sec2}, where we point out that the Coxeter diagram of $C(\Gamma)$ is the line graph~$\Line{\Gamma}$ of~$\Gamma$. In \Sref{sec3} we consider a tree $\Sigma$. We define the Coxeter diagram $\covg{\Sigma}$, with a natural projection $\Cover \ra W(\Line{\Line{\Sigma}})$, and find a system of well-defined maps $\Cover \ra C(\Sigma)$.
In \Sref{sec:RM} we rephrase the Reidemeister-Schreier algorithm, which extracts from a presentation of a group a presentation of a subgroup, for the case where the subgroup is a point stabilizer. This is utilized in \Srefs{sec5}{sec:pres} to compute our groups $C(\Gamma)_o < C(\Gamma)$. The isomorphism $C(\Sigma)_o \isom \Cover$ is established in \Sref{sec:6}.
In \Sref{sec:ex} we study the unicyclic case. Finally, in \Sref{sec8}, we show that for any graph $\Gamma$, $C(\Gamma)_o$ admits an iterated relative conjugation presentation beginning with a Coxeter group.

\subsection*{Glossary}

Throughout the paper, $\Gamma$ is a finite connected simple graph, and $\Sigma$ a tree, often a spanning subtree of $\Gamma$. The set of vertices is denoted $V(\Gamma)$, and the set of edges~$E(\Gamma)$. The line graph of $\Gamma$ (\Ssref{ssec:lg}) is denoted $\Line{\Gamma}$. The path graph with $n$ vertices is $\path[n]$; 
the cycle graph with $n$ vertices is $\cyc[n]$. We also refer to $\triangle = \cyc[3]$ and the fork graph~$\Y$. The diagram $\covg{\Sigma}$, obtained from $\Line{\Line{\Sigma}}$, the line graph of $\Line{\Sigma}$, by adding some edges, is defined in \Ssref{ss:covg}.

The Coxeter group of a Coxeter diagram $\Gamma$ is denoted $W(\Gamma)$. For our main object of study, the group $C(\Gamma)$, see \Dref{CT}; this group is endowed with a natural action on $V(\Gamma)$, and $C(\Gamma) = W(\Line{\Gamma})$ by \Cref{C=WL}. By $\GG_0$ we denote in \Sref{sec:RM} the stabilizer of the action of a group $G$; the presentation of $\GG_0(\Sigma,\Gamma)$ is given in \Pref{thepres}; and $\GG_0(\Sigma) = \GG_0(\Sigma,\Sigma)$.
The maps $\Phi_o$ and $\Psi_o$ are defined in \Ssref{ss:Phi} and \Eq{Psidef}, respectively; $\Theta$ is defined in \Pref{Thetadef}.

\smallskip

{\bf{AI transparency declaration}}. \texttt{ChatGpt 5.6-Pro} suggests, and we accept: ``Most of this paper was written before strong \texttt{AI} models were publicly available. During the final stages of revision, the authors used \texttt{ChatGPT} as a critical reader, mainly to suggest improvements in exposition and wording, to help position the work in relation to the literature, and to flag possible mathematical errors or inconsistencies. All such suggestions were independently examined by the authors, who take full responsibility for the final text and all mathematical claims.''

\section{Coxeter groups of line graphs}\label{sec2}

Let $\Gamma$ be any finite simple graph. Label the vertices by $1,\dots,n$. To an edge $e = \set{i,j}$, associate the transposition $(i\,j) \in S_n$. If $\Gamma$ is connected, then the transpositions associated to the edges generate the full symmetric group $S_n$.

\begin{thm}[\cite{Serg, Sol}]\label{Serg}
Let $\Gamma$ be a connected simple graph on $n$ vertices. Consider the group whose generators are the edges of $\Gamma$, with
five families of defining relations:
\begin{enumerate}
\item[(1)] $e^2 = 1$ for every edge $e$;
\item[(2)] $(ee')^2 = 1$ if the edges $e,e'$ are disjoint;
\item[(3)] $(ee')^3 = 1$ if the (distinct) edges $e,e'$ have a common vertex;
\item[(4)] $[e,e'e''e'] = 1$ if $e,e',e''$ share a common vertex; and
\item[(5)] $e_1\cdots e_{n-1} = e_2 \cdots e_n$ for any simple cycle $e_1,\dots,e_n$.
\end{enumerate}
Then, projecting each edge to the associated transposition, is an isomorphism of the group with $S_n$.
\end{thm}
See \cite[Section~2]{Cox} for more details on this presentation.

\smallskip
\subsection{The Coxeter group $C(\Gamma)$}

\begin{defn}\label{CT}
Let $\Gamma$ be a finite simple graph. The group generated by the edge set~$E(\Gamma)$, subject to the relations (1)--(3) of \Tref{Serg}, is denoted $C(\Gamma)$, and called the {\bf{Coxeter group~of~$\Gamma$}}.
\end{defn}

For example, if $\Gamma$ is the {\bf{path graph}} $\path$, with $n$ vertices $1$---$2$---$\,\cdots$---$n$, then $C(\Gamma) \isom S_n$ by the classical Coxeter generators, and indeed there are no instances of relations (4) or (5).

\begin{exmpl}\label{DY}
As another example, for the graphs $\triangle$ and $\Y$ depicted in \Fref{basic}, both $C(\triangle)$ and $C(\Y)$ have the presentation
$$\sg{x,y,z \subjectto x^2=y^2=z^2=(xy)^3=(yz)^3=(zx)^3=1},$$
because in each case the three edges intersect in pairs. The graph $\Y$, isomorphic to the complete bipartite graph $K_{1,3}$, is commonly known as the {\bf{claw}}.
\end{exmpl}

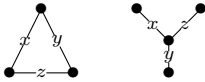
\begin{figure}
$\xymatrix@C=6pt@R=6pt{
{} & *={\bullet} \ar@{-}[ddl]|x \ar@{-}[ddr]|y & {} & {} & *={\bullet} \ar@{-}[dr]|x & {} & *={\bullet} \ar@{-}[dl]|z \\
{} & {} & {} & {} & {} & *={\bullet} \ar@{-}[d]|y & {} \\
*={\bullet} \ar@{-}[rr]|z & {} & *={\bullet} & {} & {} & *={\bullet} & {}
}$
\caption{The graphs $\triangle$ and $\Y$, with edges labelled} \label{basic}
\end{figure}

\begin{exmpl}
A Coxeter group is {\bf{odd-angled}} when the order of any product of generators $ee'$ is odd. The only connected simple graphs $\Gamma$ for which $C(\Gamma)$ is odd-angled are the {\bf{stars}} (where all edges intersect in a joint vertex), and $\triangle$.
\end{exmpl}

\begin{rem}
If $\Gamma$ is not connected, say a disjoint union of subgraphs $\Gamma = \Gamma_1 \cup \Gamma_2$, then $C(\Gamma) = C(\Gamma_1) \times C(\Gamma_2)$.
\end{rem}

\smallskip
\subsection{Line graphs}\label{ssec:lg}

Recall that for any graph $\Gamma$, the {\bf{line graph}} of $\Gamma$ is the graph~$\Line{\Gamma}$ whose {\emph{vertices}} are the edges of~$\Gamma$, and whose {\emph{edges}} are the pairs of intersecting edges in~$\Gamma$. For example, both $\Line{\triangle}$ and $\Line{\Y}$ are isomorphic to~$\triangle$. The line graph of the path $\path[4] = \xymatrix@C=16pt{*={\bullet}\ar@{-}[r]^x & *={\bullet}\ar@{-}[r]^y & *={\bullet}\ar@{-}[r]^z & *={\bullet}}$\ \ is the path
$\path[3] = \xymatrix@C=12pt{x\ar@{-}[r] & y\ar@{-}[r] & z}$, and more generally
$\Line{\path[n]} = \path[n-1]$.
The line graph of a connected graph is connected. By Whitney's graph isomorphism theorem \cite{Whit}, the line graph of a connected graph determines the graph, except for the colliding pair $\triangle,\Y$.

Beineke \cite{Bein} showed that a graph is a line graph if and only if it does not contain, as an induced subgraph, any of nine forbidden graphs, whose list is known. The list includes the claw. In particular, the only trees that can be realized as line graphs are paths.

\begin{rem}\label{!}
The Monster, and many other sporadic simple groups, are presented as quotients of Coxeter groups with graphs that are shaped as generalized~$\Y$ (see the section on the $\Y$-groups in the ATLAS, \cite[pp.~232--233]{ATLAS}). This may hint that Coxeter groups of line graphs are fundamentally more accessible than the general case.
\end{rem}

Let $\dist[\Gamma]{\cdot}{\cdot}$ denote the distance function of vertices in $\Gamma$. If $e,f$ are edges, let $\dist[\Gamma]{e}{f}$ be the minimal distance between vertices of $e$ and $f$. The following is well-known:
\begin{rem}\label{mydist}
For any two edges $e \neq f$ of $\Gamma$, viewed as vertices of $\Line{\Gamma}$, $\dist[\Line{\Gamma}]{e}{f} = \dist[\Gamma]{e}{f}+1$.
\end{rem}

\smallskip
\subsection{The Coxeter diagram of $C(\Gamma)$}

Let $S$ be the set of generators of a Coxeter group $W$. Recall that the {\bf{Coxeter diagram}} of the pair $(W,S)$ is a labelled graph on the vertex set $S$ encoding the defining relations of $W$, by connecting $s_1,s_2 \in S$ by an edge labelled by the order $o(s_1s_2)$ of $s_1s_2$ in~$W$. The elements of $S$ are assumed to be distinct, so the label $1$ is forbidden. The label $3$ is omitted, and an edge with the label $2$ is omitted altogether. Thus the Coxeter diagram is a simple unlabelled graph precisely when all the  products~$s_1s_2$ have order~$2$ or~$3$. In this case the group is called {\bf{simply-laced}}.

Recall that if $\Gamma$ is a Coxeter diagram, the associated Coxeter group is denoted by~$W(\Gamma)$. The generators of $W(\Gamma)$ correspond to the vertices of $\Gamma$, with the relations $s_i^2 = 1$ and $(s_i s_j)^{n_{ij}} = 1$ where $n_{ij}$ is the label on the edge connecting~$s_i$ and~$s_j$ (with $n_{ij} = 3$ for an edge with no label, and $n_{ij} = 2$ when there is no edge).

For example, if $\Gamma$ is any graph, one may define a Coxeter group $W(\Line{\Gamma})$ from the line graph $\Line{\Gamma}$, which would be generated by the edges of $\Gamma$ with three types of relations: $e^2 = 1$; $(ee')^2 = 1$ when the edges are not connected by an edge in $\Line{\Gamma}$; and $(ee')^3 = 1$ when the edges are connected in $\Line{\Gamma}$. But the latter conditions are equivalent to $e,e'$ being disjoint or intersecting in $\Gamma$, respectively.
\begin{cor}\label{C=WL}
Our group $C(\Gamma)$ from \Dref{CT} is the Coxeter group $W(\Line{\Gamma})$.
\end{cor}

\begin{exmpl}\label{cn} 
Let $\cyc$ denote a cycle graph on $n$ vertices (so that $\triangle = \cyc[3]$). Thus $\Line{\cyc} \isom \cyc$, and the group $C(\cyc) = W(\cyc)$ is the Coxeter group of type $\tilde{A}_{n-1}$, isomorphic to $S_n \semidirect (\Z^{n})_0$ under the action on indices, where $(\Z^n)_0$ is the submodule of $\Z^n$ of zero-sum vectors.
\end{exmpl}

We come back to this example below.

\smallskip
\subsection{Further properties of $C(\Gamma)$}

There is a classical classification of the Coxeter groups acting simplicially on the sphere, the Euclidean space, or the hyperbolic space (\eg\ \cite[Thm.~2.4.0.5]{Cash} and the tables therein). We note that $C(\Gamma)$ rarely has such an action.
\begin{prop}
Let $\Gamma$ be a connected simple graph. The group $C(\Gamma)$ acts simplicially ---
\begin{enumerate}
\item on a sphere, if and only if $\Gamma = \path[n]$ ($n\geq 2$);
\item on a Euclidean space, if and only if $\Gamma = \Y$ or $\Gamma = \cyc[n]$ ($n\geq 3$), in which case $C(\Gamma)$ has type $\tilde{A}_2$ or $\tilde{A}_{n-1}$;
\item on a hyperbolic space, never cocompactly.
\end{enumerate}
\end{prop}

By \cite[Thm.~5.5.1.1]{Cash}, we get:
\begin{cor}
Let $\Gamma$ be a connected simple graph. Then $C(\Gamma)$ is:
\begin{enumerate}
\item finite -- when $\Gamma = \path[n]$ ($n\geq 1$);
\item virtually abelian -- when $\Gamma = \Y$ or $\Gamma = \cyc[n]$ ($n \geq 3$);
\item has free nonabelian subgroups -- in any other case.
\end{enumerate}
\end{cor}

\smallskip
\subsection{Parabolic subgroups}

Let $\Gamma$ be a finite simple graph, and $\Gamma' \sub \Gamma$ any subgraph. There is an obvious map from $C(\Gamma')$ to $C(\Gamma)$, sending the generators, namely the edges of $\Gamma'$, to the same edges in $\Gamma$ as elements of $C(\Gamma)$.

\begin{prop}\label{parabolic}
Let $\Gamma' \sub \Gamma$ be any subgraph.
Then $C(\Gamma')$ is isomorphic to the subgroup of $C(\Gamma)$ generated by the edges of $\Gamma'$.
\end{prop}
\begin{proof}
By \Cref{C=WL}, $C(\Gamma) = W(\Line{\Gamma})$ and $C(\Gamma') = W(\Line{\Gamma'})$. By definition, $\Line{\Gamma'}$ is the full subgraph of $\Line{\Gamma}$ induced by $\Gamma'$. Now $C(\Gamma) = W(\Line{\Gamma})$ is a Coxeter group generated by the edges of $\Gamma$, and the subgroup given above is the parabolic subgroup generated by the edges of $\Gamma'$. Since a parabolic subgroup is a Coxeter group \cite[1.10]{Humph}, $\sg{E(\Gamma')}$ is the Coxeter group defined by the Coxeter diagram associated to~$\Gamma'$, which is $\Line{\Gamma'}$; therefore $\sg{E(\Gamma')} = W(\Line{\Gamma'}) = C(\Gamma')$.
\end{proof}

We can thus define a {\bf{standard parabolic subgroup}} of $C(\Gamma)$ to be any subgroup generated by the edges of a subgraph. Adapting this proof to \cite[1.13]{Humph}, we also obtain:
\begin{prop}
The lattice of standard parabolic subgroups of $C(\Gamma)$ is isomorphic to the lattice of edge-subgraphs of~$\Gamma$.
\end{prop}

(On passing, we note that the commensurator of a parabolic subgroup in any Coxeter group is parabolic, \cite{Paris}.)

\begin{cor}\label{conji}
Let $\Gamma$ be a finite connected simple graph. Then any two edges of~$\Gamma$ are contained in a subgroup of~$C(\Gamma)$ isomorphic to some~$S_m$. In particular all the edges of~$\Gamma$ are conjugate in~$C(\Gamma)$.
\end{cor}
\begin{proof}
Let $x,y$ be edges. Since $\Gamma$ is connected, there is a simple path $x = u_1, \dots, u_t = y$ in $\Gamma$. By \Pref{parabolic}, $\sg{u_1,\dots,u_t} \isom S_{t+1}$, a group in which $u_1,u_t$ are transpositions and thus conjugate. 
\end{proof}

There is a characterization of the finite subgroups of a Coxeter group \cite[Thm~3.2.4.17]{Cash}. For $C(\Gamma)$, these results can be formulated as follows.
\begin{prop}\label{maxfin}
A maximal finite subgroup of $C(\Gamma)$ is a direct product of symmetric groups. Up to conjugation, every such group is generated by a disjoint union of paths in $\Gamma$ (namely the sets of vertices are disjoint).
\end{prop}

\begin{exmpl}
For $\Sigma = \raisebox{2.0ex}{\xymatrix@C=12pt@R=12pt\CoxeterDiagIII}$\ , the maximal finite subgroups of $C(\Sigma)$ are isomorphic to $S_4$ and $S_3 \times S_3$.
\end{exmpl}

In this direction, note that since $C(\Gamma)$ is simply-laced, the $1$-skeleton of the nerve of~$C(\Gamma)$ is the complete graph, and so $C(\Gamma)$ is always $0$- or $1$-ended \cite[Thm~8.7.2]{Davis}.
Another important property shared by some Coxeter groups is hyperbolicity. By \cite[Theorem~5.4.2.1]{Cash}, a Coxeter group $W$ is hyperbolic if and only if its Davis complex is $\opname{CAT}(-1)$, if and only if $\Z^2 \not\hookrightarrow W$.

\begin{prop}
No infinite group of the form $C(\Gamma)$ is hyperbolic.
\end{prop}
\begin{proof}
Let $\Gamma$ be a connected simple graph. Assume there is no embedding $\Z^2 \hra C(\Gamma)$. Note that $C(\Y)$ and $C(\cyc[n])$ are of types $\tilde{A}_2$ or $\tilde{A}_{n-1}$, respectively. By the structure of parabolic subgroups in \Pref{parabolic}, and \Erefs{cn}{DY2} (below), $\Gamma$ has no subgroups of the form $\Y$ or $\cyc[n]$ ($n \geq 3$), so $\Gamma$ is a tree with vertices of degree at most~$2$. Thus $\Gamma = \path[n]$ for some~$n$, in which case $C(\Gamma) \isom S_{n+1}$ is finite.
\end{proof}

\smallskip
\subsection{Transpositional action}\label{ssec:ta}

Let $n$ be the number of vertices of the graph $\Gamma$. Mapping an edge of~$\Gamma$ to the associated transposition induces a projection $C(\Gamma) \ra S_n$, in other words an action of $C(\Gamma)$ on the~$n$ vertices. We say that an action of a group~$G$, generated by a set $S$, is {\bf{transpositional}} if the generators in $S$ act
 as transpositions. To such an action we associate the graph on the vertex set $X$, whose edges are those associated to the generators viewed as transpositions.

Fix a simple graph $\Gamma$. Let $\CAT(\Gamma)$ be the category of Coxeter groups acting transpositionally via $\Gamma$, defined as follows. The objects are the Coxeter groups generated by the edge set $E(\Gamma)$ (each with its own set of relations) and acting transpositionally on the vertex set of $\Gamma$, such that the associated graph is $\Gamma$; and the morphisms are the group homomorphisms induced by the identity map on the common generating set $E(\Gamma)$, when such a homomorphism is well-defined. For example, recall the triangle graph $\triangle$. The objects of $\CAT(\triangle)$ are the triangle groups
$$G_{3m,3n,3k} = \sg{a,b,c \suchthat\ a^2=b^2=c^2=1,\ (ab)^{3m}=(bc)^{3n}=(ca)^{3k} = 1}$$
for arbitrary $m,n,k$, with a morphism $G_{3m,3n,3k} \ra G_{3m',3n',3k'}$ whenever $m'\divides m$, $n'\divides n$ and $k'\divides k$. The group $C(\triangle) = G_{3,3,3}$ is the terminal object in this category.

\begin{thm}
The group $C(\Gamma)$ is the terminal object of $\CAT(\Gamma)$.
\end{thm}
\begin{proof}
Let $G$ be any object in $\CAT(\Gamma)$. We need to verify that the identity map on $E(\Gamma)$ induces a well-defined homomorphism $G \ra C(\Gamma)$. Since $G$ is a Coxeter group, it is defined by relations of the form $(uv)^n = 1$, where $u,v$ are edges in $\Gamma$.
In $S_n$, the product of disjoint transpositions is of order $2$, and the product of adjacent transpositions is a 3-cycle. If $u,v$ are disjoint then their images in $S_n$ commute, so $n$ must be even; but the relation $(uv)^2 = 1$ holds in $C(\Gamma)$, so $(uv)^n = 1$ is satisfied by the morphism. Similarly if $u,v$ share a vertex then the product of their images in $S_n$ has order $3$, so $3 \divides n$; but the relation $(uv)^3 = 1$ holds in $C(\Gamma)$ so $(uv)^n = 1$ is satisfied in this case as well. This shows that the relations defining $G$ are implied by the relations defining $C(\Gamma)$, so the map is well-defined.
\end{proof}

\smallskip
\subsection{The groups $\CY(\Gamma)$}

The group $C(\Gamma)$ is defined by the relations (1)--(3) of \Tref{Serg}. Adding the relations~(4), of the form $[e,e'e''e']=1$, we obtain a quotient group, denoted by $\CY(\Gamma)$, which was fully described in \cite{Cox} as a semidirect product $S_n \ltimes F_{t,n}$, where $t$ is the rank of the fundamental group $\pi_1(\Gamma)$ and $n$ the number of vertices. If $\Gamma$ is a tree then the fifth family of relations is empty, in which case $\CY(\Gamma) = S_n$. It is shown in \cite{Cox} that if $\Gamma$ is not a tree then the projection $\CY(\Gamma) \ra S_n$ is proper. In fact, $\CY(\Gamma)$ is virtually abelian if~$\Gamma$ has a single cycle, and~$\CY(\Gamma)$ contains a free nonabelian group if~$\Gamma$ has more than one cycle. The action of~$C(\Gamma)$ on the vertices is induced by the action of $\CY(\Gamma)$, so there are natural maps
\begin{equation}\label{maps}
\xymatrix{C(\Gamma) \ar@{->}[r] & \CY(\Gamma) \ar@{->}[r] & S_n.}
\end{equation}

\begin{exmpl}\label{DY2}
Let us compare the groups $C(\cdot)$ and $\CY(\cdot)$ for the graphs $\triangle$ and~$\Y$ of \Fref{basic}. Combining \eq{maps} for the two graphs, we have the following commuting diagram,
$$\xymatrix@R=12pt{
C(\Y) \ar@{=}[d] \ar@{->}[r] & \CY(\Y) \ar@{=}[r] & S_4 \ar@{->}[d] \\
C(\triangle) \ar@{=}[r] & \CY(\triangle) \ar@{->}[r] & S_3
}
$$
where the mappings are as follows. The equality $C(\Y) = C(\triangle)$ is formal, as described in \Eref{DY}. The equality $C(\triangle) = \CY(\triangle)$ is also formal, because~$\triangle$ has no vertices of degree $\geq 3$. The equality $\CY(\Y) = S_4$ holds as explained above because $\Y$ has no cycles. Although $C(\Y) = C(\triangle)$, the natural actions of these groups differ, because the number of vertices of the two graphs differ. We now make \Eref{cn} explicit, using the notation of \Eref{DY}. Letting $\alpha = xyzy$, $\beta = yzxz$ and $\gamma = zxyx$ in $C(\Y)$, it is not hard to verify that $\alpha\beta\gamma = 1$ and $\sg{\alpha,\beta,\gamma} \isom \Z^2$, from which we obtain the commuting diagram with short exact sequences
$$\xymatrix@R=12pt{
1 \ar@{->}[r] & \sg{\alpha^2,\beta^2,\gamma^2} \ar@{^(->}[d] \ar@{->}[r] & C(\Y) \ar@{=}[d] \ar@{->}[r]  & S_4 \ar@{->}[d] \ar@{->}[r] & 1 \\
1 \ar@{->}[r] & \sg{\alpha,\beta,\gamma} \ar@{->}[r] & C(\triangle) \ar@{->}[r] & S_3 \ar@{->}[r] & 1
}$$
In particular $C(\triangle) \isom S_3 \ltimes \Z^2$, as we can take $S_3 \stackrel{\sim}{\lra} \sg{x,y} \leq C(\triangle)$ to split the bottom row. Notice that $\alpha^2 = [x,yzy]$, so $\alpha^2 = 1$ is one of the defining relations of $\CY(\Y)$, which reproves the equality $\CY(\Y) = S_4$. In particular, the projection $C(\Y) \ra \CY(\Y)$ is proper.
\end{exmpl}

\begin{rem}
The extension $1 \lra \Z^2 \lra C(\Y) \lra S_4 \lra 1$ does not split. This is because $C(\Y) \isom \Z^2 \rtimes S_3$ has no elements of order $4$. Indeed if $(a \s)^4 = 1$, for $a \in \Z^2$ and $\s \in S_3$, then $\s^4 = 1$ implies $\s^2 = 1$, but then $(a \s)^4 = (a\;\s\!\,a)^2$, so $a\;\s\!\,a = 1$ as $\Z^2$ is torsion-free, and then $(a \s)^2 = 1$. Alternatively, by \Pref{maxfin} all the maximal finite subgroups of $C(\Y)$ are isomorphic to $S_3$.
\end{rem}

\begin{rem}
Fischer \cite{Fischer} defined a ``$3$-transposition group'' as a group generated by a conjugacy class $I$ of involutions such that $o(ss') \leq 3$ for any two elements $s,s' \in I$. It is true that $C(\Gamma)$ is generated by conjugate (by \Cref{conji}) involutions satisfying this condition on the order; but the order condition does not hold for {\it{all}} the conjugates. Indeed, even in $C(\triangle)$, the element $\alpha = x\cdot yzy$ has infinite order.
\end{rem}

\begin{prop}
For a simple graph $\Gamma$, the projection $C(\Gamma) \ra \CY(\Gamma)$ is proper, unless $\Gamma$ is a cycle or the path graph.
\end{prop}
\begin{proof}
If $\Gamma$ is the cycle or the path graph, there are no claw subgraphs and no new relations, so $C(\Gamma) = \CY(\Gamma)$. Now assume $\Gamma$ is not a cycle or a path; then $\Gamma$ contains~$\Y$ as a subgraph. By \Pref{parabolic} and the analogous statement for $\CY(\cdot)$ given in \cite[Prop.~7.8]{Cox}, we have the commutative diagram
$$\xymatrix@R=14pt@C=14pt{ C(\Y) \ar@{^(->}[r] \ar@{->}[d] & C(\Gamma) \ar@{->}[d] \\ \CY(\Y) \ar@{^(->}[r] & \CY(\Gamma) }
$$
The final statement in \Eref{DY2} is that the projection in the left-hand side is proper, so the claim follows.
\end{proof}

\section{Techniques for trees}\label{sec3}

In this section we study the Coxeter group $C(\Sigma)$ where $\Sigma$ is a tree. The group~$C(\Sigma)$ acts transpositionally on the vertex set $\Omega = V(\Sigma)$.

\smallskip
\subsection{Edges in $\Line{\Sigma}$}

For any pair of intersecting edges $\set{a,b},\set{b,c} \in E(\Sigma)$ (where $a,b,c$ are distinct vertices of $\Sigma$), we denote the corresponding edge of $\Line{\Sigma}$ by $$[abc] = \set{\set{a,b},\set{b,c}}.$$
By definition $[abc] = [cba]$.


By \Dref{CT}, the group $C(\Line{\Sigma})$ is generated by all the edges $[abc]$ of $\Line{\Sigma}$. Notice that $[abc]$ and $[a'b'c']$ are disjoint in $\Line{\Sigma}$ if and only if the sets of vertices $\set{a,b,c}$ and $\set{a',b',c'}$ in $\Sigma$ intersect at most once, because $\Sigma$ is a tree (so the case $\set{a,c} = \set{a',c'}$ with $b' \neq b$ is impossible).

Since the tree~$\Sigma$ is connected, $\Line{\Sigma}$ is connected as well. Let us inspect the path connecting two edges of $\Line{\Sigma}$ as a substructure of $\Sigma$ itself.
We say that disjoint edges $[abc]$ and $[a'b'c']$ in $\Line{\Sigma}$ are {\bf{evenly connected}} if
the path in~$\Sigma$ connecting the middle points $b$ and $b'$ does not pass through any of the vertices $a,c,a',c'$ (this is Case (a) in the top row of \Fref{TLT}). This relation is clearly symmetric.

\begin{prop}\label{graph1}
For two disjoint edges $[abc],[a'b'c']$ of $\Line{\Sigma}$, the following are equivalent.
\begin{enumerate}
\item $[abc]$ and $[a'b'c']$ are evenly connected.
\item The four distances
$$\dist[\Sigma]{a}{a'}, \quad \dist[\Sigma]{a}{c'}, \quad \dist[\Sigma]{c}{a'}, \quad \dist[\Sigma]{c}{c'}$$
in $\Sigma$ are equal.
\item Writing $[abc] = \set{x,y}$ and $[a'b'c'] = \set{x',y'}$, the four distances
$$\dist[\Line{\Sigma}]{x}{x'}, \quad \dist[\Line{\Sigma}]{x}{y'}, \quad \dist[\Line{\Sigma}]{y}{x'}, \quad \dist[\Line{\Sigma}]{y}{y'}$$
in $\Line{\Sigma}$ are equal.
\end{enumerate}
\end{prop}
\begin{proof}
Let $[abc]$ and $[a'b'c']$ be edges in $\Line{\Sigma}$; each of these is a pair of edges in $\Sigma$, whose ``leaves'' are $a,c$ and $a',c'$, respectively. Since $\Sigma$ is a tree, there is a unique path in $\Sigma$ connecting $b$ and $b'$.

The three diagrams in the top row of \Fref{TLT} depict this structure, with a dotted line for the path, in three possible cases, characterized by the number of leaves through which the path passes. The diagrams in the bottom row depict the same structures in $\Line{\Sigma}$. It is now easy to verify, using \Rref{mydist}, that the conditions stated in the proposition hold in case (a) but not in cases (b) or (c).
\begin{figure}
$$
\xymatrix@R=2pt@C=2pt{
*={\bullet} \ar@{-}[rrdd] & & & & & & & & & & *={\bullet}\ar@{-}[lldd] \\
 & & & & & & & & & & \\
 & & *={\bullet} \ar@{-}[lldd] \ar@{.}[rrrrrr] & & & & & & *={\bullet} \ar@{-}[rrdd]& & \\
 & & & & & & & & & & \\
*={\bullet} & & & & & & & & & & *={\bullet}  }
\qquad
\xymatrix@R=2pt@C=2pt{
*={\bullet} \ar@{-}[rrdd] & & & & & & & & & & *={\bullet}\ar@{-}[lldd] \\
 & & & & & & & & & & \\
 & & *={\bullet} \ar@{-}[lldd] \ar@{.}[rrrrrrrruu] & & & & & & *={\bullet} \ar@{-}[rrdd]& & \\
 & & & & & & & & & & \\
*={\bullet} & & & & & & & & & & *={\bullet}  }
\qquad
\xymatrix@R=2pt@C=2pt{
*={\bullet} \ar@{-}[rrdd] \ar@{.}[rrrrrrrrrr] & & & & & & & & & & *={\bullet}\ar@{-}[lldd] \\
 & & & & & & & & & & \\
 & & *={\bullet} \ar@{-}[lldd]  & & & & & & *={\bullet} \ar@{-}[rrdd]& & \\
 & & & & & & & & & & \\
*={\bullet} & & & & & & & & & & *={\bullet}  }
$$
$$\xymatrix@R=2pt@C=2pt{
 & & & & & & & & & &  \\
 & *={\circ} \ar@{.}[rrd] \ar@{-}[dd] & & & & & & & & *={\circ}\ar@{.}[lld] \ar@{-}[dd] &\\
 & & & *={} \ar@{.}[rrrr] & & & &  *={} \ar@{.}[rrd] & & & \\
 & *={\circ} \ar@{.}[rru] & & & & & & & &  *={\circ} & \\
 & & & & & & & & & &  }
 \qquad
 \xymatrix@R=2pt@C=2pt{
 & & & & & & & & & &  \\
 & *={\circ} \ar@{.}[rrd] \ar@{-}[dd] & & & & & & & *={}\ar@{.}[r] & *={\circ} \ar@{-}[dd] &\\
 & & & *={} \ar@{.}[rrrrru] & & & &   & & & \\
 & *={\circ} \ar@{.}[rru] & & & & & & & &  *={\circ} & \\
 & & & & & & & & & &  }
 \qquad
 \xymatrix@R=2pt@C=2pt{
 & & & & & & & & & &  \\
 & *={\circ}  \ar@{-}[dd] \ar@{.}[r] & *={} \ar@{.}[rrrrrr] &  & & & & &*={}  \ar@{.}[r]& *={\circ} \ar@{-}[dd] &\\
 & & &  & & & &  *={}  & & & \\
 & *={\circ} & & & & & & & &  *={\circ} & \\
 & & & & & & & & & &  }
$$
$$(a) \qquad\qquad\qquad\qquad\quad (b) \qquad\qquad\qquad\qquad\quad (c)$$
\caption{Pairs of edges in $\Line{\Sigma}$ and the distances between leaves: in $\Sigma$ (top) and in $\Line{\Sigma}$ (bottom). The edges are evenly connected in case (a).}\label{TLT}
\end{figure}
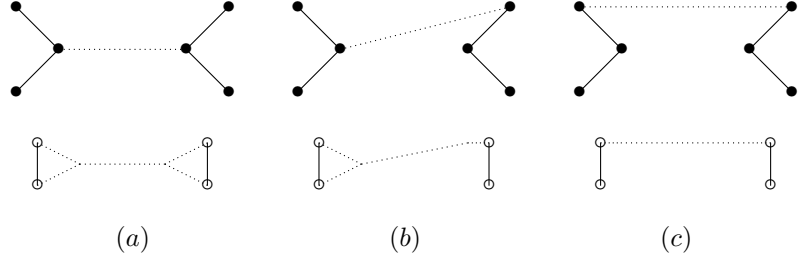
\end{proof}

\begin{cor}
The property of two edges in $\Line{\Sigma}$ to be evenly connected is intrinsic to $\Line{\Sigma}$, and does not require the original graph~$\Sigma$.
\end{cor}

\smallskip
\subsection{The Coxeter diagram $\covg{\Sigma}$}\label{ss:covg}

Our analysis of $C(\Sigma)$, where $\Sigma$ is any tree, requires the group $C(\Line{\Sigma})$, which by \Cref{C=WL} is the Coxeter group~$W(\Line{\Line{\Sigma}})$. We will define a Coxeter diagram $\covg{\Sigma}$ built from $\Line{\Line{\Sigma}}$, so that there is a natural projection $\Cover[\Sigma] \ra W(\Line{\Line{\Sigma}})$. By definition, the vertices of $\Line{\Line{\Sigma}}$ are the edges of~$\Line{\Sigma}$.

\begin{defn}\label{covgdef}
Let $\Sigma$ be a tree. The Coxeter diagram $\covg{\Sigma}$ is the graph whose vertices are the vertices of $\Line{\Line{\Sigma}}$, with edges of two types: (1) the edges of $\Line{\Line{\Sigma}}$, and (2) an edge, labelled~$\infty$, connecting every disjoint pair $[abc],[a'b'c'] \in E(\Line{\Sigma})$ of evenly connected edges.
\end{defn}

Thus $\covg{\Sigma}$ is the union of the unlabelled graph $\Line{\Line{\Sigma}}$, with new edges all labelled~$\infty$.

\begin{exmpl}
\begin{enumerate}
\item For the path graph $\Sigma = \path[n]$ we have $\Line{\path[n]} = \path[n-1]$ and $\covg{\path[n]} = \Line{\Line{\path[n]}} = \path[n-2]$.
\item For $\Sigma = \Y$ we have $\Line{\Y} = \triangle$ and $\covg{\Y} = \Line{\Line{\Y}} = \triangle$.
\end{enumerate}
\end{exmpl}

A {\bf{generalized $\Y$}} is a tree composed of three paths joined in a central point. A tree is a generalized $\Y$ if and only if it has~$3$ leaves, if and only if it has a unique vertex of degree~$3$ and no vertices of higher degree.
\begin{rem}\label{YY+}
$\covg{\Sigma} = \Line{\Line{\Sigma}}$ if and only if $\Sigma$ is a path graph or a generalized $\Y$.
\end{rem}
\begin{proof}
It is easy to realize a pair of evenly connected edges in $\Line{\Sigma}$ if there is a vertex of degree at least $4$ (where the dotted path in \Fref{TLT}(a) is empty), or if there are two vertices of degree $3$.
\end{proof}

\begin{exmpl}\label{e36}
Let $\Sigma$ be the tree at the left-hand side in \Fref{1inf}. The line graph~$\Line{\Sigma}$ is in the middle, and its line graph $\Line{\Line{\Sigma}}$ is at the right-hand side, with an extra edge $\xymatrix{[123] \ar@{-}[r]^{\infty} & [456]}$ for the pair of evenly connected disjoint edges of $\Line{\Sigma}$, thus composing $\covg{\Sigma}$.
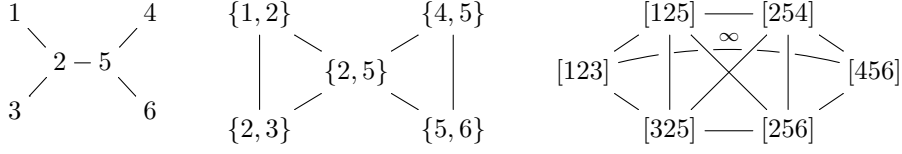
\begin{figure}
$$\xymatrix@R=6pt@C=6pt{
1 \ar@{-}[rd] & & & 4 \ar@{-}[ld] \\
{} & 2  \ar@{-}[r] \ar@{-}[ld]  & 5 \ar@{-}[rd]  \\
3 & & & 6
}
\qquad
\xymatrix@R=6pt@C=6pt{
\set{1,2} \ar@{-}[dd] \ar@{-}[rd] & & \set{4,5} \ar@{-}[dd] \ar@{-}[ld]\\
{} & \set{2,5}  \ar@{-}[ld] \ar@{-}[rd] &  \\
\set{2,3} & & \set{5,6}
}
\qquad
\xymatrix@R=6pt@C=6pt{
{} & [125] \ar@{-}[rr] \ar@{-}[rrdd] \ar@{-}[dd] & {} & [254]  \ar@{-}[dd] & {} \\
{}[123] \ar@/^2ex/@{-}[rrrr]|(0.285)\hole|(0.42)\hole|(0.57)\hole|(0.70)\hole^\infty \ar@{-}[ru] \ar@{-}[rd] & {} &  & {} & {}[456] \ar@{-}[lu] \ar@{-}[ld] \\
{} & [325] \ar@{-}[rr] \ar@{-}[rruu] & {} &  [256] & \\
}
$$
\caption{The graphs of \Eref{e36}} \label{1inf}
\end{figure}
\end{exmpl}

\begin{exmpl}\label{e37}
\Fref{2inf} provides another example similar to \Eref{e36}.
\end{exmpl}
\begin{figure}
$$\xymatrix@R=6pt@C=10pt{
*={\bullet} \ar@{-}[rd] & &&& & *={\bullet} \ar@{-}[ld] \\
{} & *={\bullet}  \ar@{-}[r] \ar@{-}[ld] &*={\bullet} \ar@{-}[r] & *={\bullet} \ar@{-}[r] & *={\bullet} \ar@{-}[rd]  \\
*={\bullet} & &&& & *={\bullet}
}
\qquad
\xymatrix@R=6pt@C=10pt{
*={\bullet} \ar@{-}[dd] \ar@{-}[rd] &&& & *={\bullet} \ar@{-}[dd] \ar@{-}[ld]\\
{} & *={\bullet}  \ar@{-}[ld] \ar@{-}[r] &*={\bullet} \ar@{-}[r] & *={\bullet} \ar@{-}[rd] & \\
*={\bullet} && && *={\bullet}
}
\qquad
\xymatrix@R=6pt@C=10pt{
{} & *={\bullet} \ar@{-}[rd] \ar@{-}[dd] & {} && *={\bullet}  \ar@{-}[dd] & {} \\
*={\bullet} \ar@/^2ex/@{-}[rrrrr]|(0.2)\hole|(0.3)\hole|(0.7)\hole|(0.8)\hole^\infty \ar@{-}[ru] \ar@{-}[rd] & {} & *={\bullet} \ar@{-}[r] & *={\bullet}  \ar@{-}[ru] \ar@{-}[rd] && *={\bullet} \ar@{-}[lu] \ar@{-}[ld] \\
{} & *={\bullet}  \ar@{-}[ru] & {} &&  *={\bullet} & \\
}
$$
\caption{The graphs of \Eref{e37}} \label{2inf}
\end{figure}
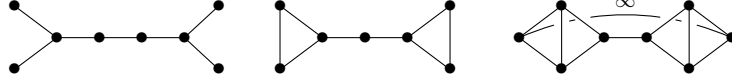

\smallskip
\subsection{The Coxeter group $\Cover$}

Having constructed in \Dref{covgdef} the Coxeter diagram $\covg{\Sigma}$ for the tree $\Sigma$, we have the associated Coxeter group~$\Cover$. It is as if this group is obtained from $C(\Line{\Sigma}) = W(\Line{\Line{\Sigma}})$ by removing the commutativity relation of each pair of evenly connected edges $[abc],[a'b'c'] \in E(\Line{\Sigma})$.
By definition, $\Cover$ is a Coxeter group, with orders~$2$, $3$ and $\infty$. Moreover, there is a natural projection
$$\Cover \ra C(\Line{\Sigma}),$$
mapping each edge $[abc]$ to itself. Indeed every relation assumed in $\Cover$ is already present in $C(\Line{\Sigma})$.

Restating \Rref{YY+}, we note:
\begin{rem}\label{Y+}
The projection $\Cover \ra C(\Line{\Sigma})$ is an isomorphism if and only if~$\Sigma$ is a path graph or a generalized~$\Y$.
\end{rem}

\smallskip
\subsection{The system of paths in $\Sigma$}\label{ss:pi}

For any two vertices $a,b$ of the tree $\Sigma$, let $u_1,\dots,u_t$ be the unique path from $a$ to $b$, as in \Fref{path}.
\begin{figure}
$$\xymatrix{
a \ar@{-}[r]^{u_1} & \circ \ar@{-}[r]^{u_2} & \circ \ar@{. }[r] & \circ \ar@{-}[r]^{u_t} & b
}$$
\caption{The path from $a$ to $b$ in $\Sigma$} \label{path}
\end{figure}
We set
\begin{equation}\label{pidef}
\pi_{ab} = u_1 \cdots u_t,
\end{equation}
viewed as an element of $C(\Sigma)$. In particular $\pi_{aa} = 1$ is the identity element. 
In terms of the action of $C(\Sigma)$ on the vertices, note that \begin{equation}\label{piact}
\pi_{ab}(b) = a,\end{equation}
by induction on $t$.

\begin{prop}\label{pitrans}
For any $a,b,c \in V(\Sigma)$, we have that $\pi_{ab} \pi_{bc} = \pi_{ac}$.
\end{prop}
\begin{proof}
First observe that $\pi_{ab}\pi_{ba} = 1$, because the edges in $\Sigma$ are elements of order~$2$ in $C(\Sigma)$. Next, assume $b$ is a vertex between $a$ and $c$. Then the path from $a$ to $c$ is the concatenation of the paths from $a$ to $b$ and then from $b$ to $c$, and the claim is immediate. Finally, let $b'$ be the vertex on the path from $a$ to $c$ that is closest to~$b$, where perhaps $b' \in \set{a,c}$. Then $\pi_{ab} = \pi_{ab'}\pi_{b'b}$ and $\pi_{bc} = \pi_{bb'}\pi_{b'c}$, so $\pi_{ab}\pi_{bc} = \pi_{ab'}\pi_{b'b}\pi_{b'b}\pi_{b'c} = \pi_{ab'}\pi_{b'c} = \pi_{ac}$ by the previous case.
\end{proof}


\smallskip
\subsection{The maps $\Phi_o \co \Cover \ra C(\Sigma)$}\label{ss:Phi}

We define a system of maps $$\Phi_o \co \Cover \ra C(\Sigma),$$
one for each vertex $o \in V(\Sigma)$, directly from the presentation, see \Fref{thegroups}. We will later prove that $\Phi_o$ is an isomorphism of $\Cover$ with the stabilizer $C(\Sigma)_o \sub C(\Sigma)$.

\begin{figure}
$$\xymatrix@C=8pt{
{} & \Cover \ar@{->}[rrrr]^{\Phi_o} \ar@{->>}[d] &&& & C(\Sigma)_o \ar@{^(->}[d] & {} \\
W(\Line{\Line{\Sigma}}) \ar@{=}[r] &  C(\Line{\Sigma}) &&& & C(\Sigma) \ar@{=}[r] & W(\Line{\Sigma})
}$$
\caption{Coxeter groups, the stabilizer and the cover}\label{thegroups}
\end{figure}

\begin{prop}\label{abc}
Fix a vertex $o$. Let $[abc]$ be an edge of $\Line{\Sigma}$.
Then $$\pi_{oa} \pi_{bc} \pi_{oa}^{-1} = \pi_{oc} \pi_{ab} \pi_{oc}^{-1}$$ in $C(\Sigma)$.
\end{prop}
\begin{proof}
Let $x = \set{a,b}$ and $y = \set{b,c}$ be the edges in $\Sigma$ composing the edge~$[abc]$ in~$\Line{\Sigma}$. As elements of $C(\Sigma)$, $x = \pi_{ab}$ and $y = \pi_{bc}$, and by definition $\pi_{ac} = xy$. Now
$\pi_{oc} x \pi_{oc}^{-1} = \pi_{oa}\pi_{ac}x\pi_{ac}^{-1}\pi_{oa}^{-1} = \pi_{oa} xyxyx \pi_{oa}^{-1} = \pi_{oa} y \pi_{oa}^{-1}$
because $(xy)^3 = 1$.
\end{proof}

\begin{prop}\label{map}
For a fixed vertex $o$, the map $\Phi_o \co \Cover \ra C(\Sigma)$ defined on the generators by
\begin{equation}\label{Phidef}
\Phi_o([abc]) = \pi_{oa} \pi_{bc} \pi_{oa}^{-1} 
\end{equation}
is well-defined.
\end{prop}
\begin{proof}
We need to verify that $\Phi_o$ respects the defining relations of $\Cover$. Notice that if $o,o'$ are two vertices, then $\Phi_{o'}(\cdot) = \pi_{o'o} \Phi_{o}(\cdot) \pi_{o'o}^{-1}$ in their action on generators. This shows that the choice of origin $o$ is immaterial. 
\begin{enumerate}
\item For any edge $[abc]$, $\Phi_o([abc])^2 = (\pi_{oa} y \pi_{oa}^{-1})^2 = \pi_{oa} y^2 \pi_{oa}^{-1} = 1$ where $y = \pi_{bc}$.
\item Let $[abc]$ and $[a'b'c']$ be disjoint edges in $\Line{\Sigma}$. If they are evenly connected there is nothing to prove, so assume otherwise. We need to prove that $[\pi_{oa} \pi_{bc} \pi_{oa}^{-1}, \pi_{oa'} \pi_{b'c'} \pi_{oa'}^{-1}] = 1$. Conjugating, we need to prove that
    $[\pi_{bc} , \pi_{aa'} \pi_{b'c'} \pi_{aa'}^{-1}] = 1$. Recall that the top line of
    \Fref{TLT} depicts the possible relative positions of the edges $[abc]$ and $[a'b'c']$, classified by the number of leaves through which the path from $b$ to $b'$ passes. Case~(a) is ruled out by assumption, so we are in cases~(b) or~(c),
and by symmetry we assume they are labelled as follows.
$$
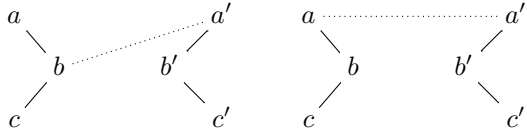
\xymatrix@R=0pt@C=0pt{
a \ar@{-}[rrdd] & & & & & & & & & & a'\ar@{-}[lldd] \\
 & & & & & & & & & & \\
 & & b \ar@{-}[lldd] \ar@{.}[rrrrrrrruu] & & & & & & b' \ar@{-}[rrdd]& & \\
 & & & & & & & & & & \\
c & & & & & & & & & & c' }
\qquad
\xymatrix@R=0pt@C=0pt{
a \ar@{-}[rrdd] \ar@{.}[rrrrrrrrrr] & & & & & & & & & & a'\ar@{-}[lldd] \\
 & & & & & & & & & & \\
 & & b \ar@{-}[lldd]  & & & & & & b' \ar@{-}[rrdd]& & \\
 & & & & & & & & & & \\
c & & & & & & & & & & c'  }
$$
In both cases the path from $a$ to $a'$ is disjoint from $\set{b',c'}$, so $\pi_{aa'}\pi_{b'c'}\pi_{aa'}^{-1} = \pi_{b'c'}$. It remains to show that
  $[\pi_{bc} , \pi_{b'c'}] = 1$, but the respective edges are disjoint ($c \neq c'$ by the choice of labels, and $b = b'$ is only possible in Case (a), when the dotted path is empty).
\item Let $[abc]$ and $[a'b'c']$ be a pair of intersecting edges in $\Line{\Sigma}$. Two pairs of intersecting edges in $\Sigma$ can collide in a common edge in two ways (up to symmetry), as depicted below. Namely, either $a' = b$ and $b' = c$, or $b' = b$ and $a' = a$, with the respective subgraphs in $\Sigma$ being as follows:
    $$
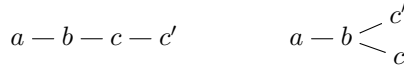
\xymatrix@R=-5pt@C=8pt{
{} & {} & {} & {} & {} & {} &{} &{} & c'  \\
a \ar@{-}[r]  & b \ar@{-}[r] & c\ar@{-}[r] & c' & {} &{} &a \ar@{-}[r]  & b \ar@{-}[rd] \ar@{-}[ru] &  & \\
{} & {} & {} & {} &{} &{} & {} & {} & c }
    $$
We need to prove that
    $(\pi_{oa}\pi_{bc} \pi_{oa}^{-1} \cdot \pi_{oa'}\pi_{b'c'} \pi_{oa'}^{-1})^3 = 1$. Equivalently, that $(\pi_{bc} \pi_{aa'}\pi_{b'c'} \pi_{a'a})^3 = 1$.

    In the first case, we want $(\pi_{bc} \pi_{ab}\pi_{cc'} \pi_{ba})^3 = 1$, but $\pi_{bc} \pi_{ab}\pi_{cc'} \pi_{ba} = \pi_{bc} \pi_{cc'}$ because the paths from $a$ to $b$ and from $c$ to $c'$ are disjoint, each being a single edge, and this element has order $3$ in $C(\Sigma)$ by definition.
    In the second case, the computation is $(\pi_{bc} \pi_{aa}\pi_{bc'} \pi_{aa})^3 = (\pi_{bc}\pi_{bc'})^3 = 1$, for the same reason.
\end{enumerate}
\end{proof}

\begin{rem}\label{map0}
For any vertex $o$, the image of $\Phi_o \co \Cover \ra C(\Sigma)$ is contained in the stabilizer of the vertex $o$. Indeed, the image of each generator $[abc]$ acts as $\pi_{oa}\pi_{bc}\pi_{ao}(o) = \pi_{oa}\pi_{bc}(a) = \pi_{oa}(a) = o$, since $\pi_{bc}$ transposes $b,c$ and leaves every other vertex, including $a$, fixed.
\end{rem}

\section{Reidemeister-Schreier method for stabilizers}\label{sec:RM}

In this section we are concerned with an arbitrary group $G = \sg{X \subjectto R}$, presented by generators and relations.
The Reidemeister-Schreier method (see \cite[Section~2.3]{MKS}, \cite[Subsection~1.3.7]{CGKZ}) 
is tasked with producing a presentation of a subgroup of $G$, given by a set of generators as words in the free group $\sg{X}$.
We present an adaptation to the classical method when $G$ is endowed with a transitive action on a set~$\Omega$, and~$H$ is the point stabilizer.

For simplicity, we assume that the given presentation of $G$ is {\bf{positive}} (the relations have the form $w = 1$ where $w$ is a word without inverses). When $G$ is generated by elements of finite order, as is the case for quotients of Coxeter groups, every presentation can be made positive (by replacing any $x^{-1}$ by $x^{o(x)-1}$).

\begin{algo}[Reidemeister-Schreier for stabilizers]\label{RSalg}
Let $\sg{X\subjectto R}$ be a presentation of a group $G$, acting transitively on a set $\Omega$. Fix a base point $o \in \Omega$. The process computes a presentation for the stabilizer $G_o < G$.

The action of $G$ on $\Omega$ induces an action of the free group $\sg{X}$ on $\Omega$. We say that $\set{\s_\omega \suchthat \omega \in \Omega} \subset \sg{X}$ is a {\bf{Schreier set}} if \begin{equation}\s_\omega(\omega) = o\end{equation} for each $\omega \in \Omega$, and the set is closed under taking initial subwords. Fix a Schreier set $S$ (which always exists). 

Consider the pairs $(\omega,x)$ for $\omega \in \Omega$ and $x \in X$. The pair $(\omega, x)$ is {\bf{redundant}} if~$\s_\omega x$ is in the Schreier set $S$. Equivalently, if~$\s_\omega x = \s_{x^{-1}(\omega)}$ in the free group~$\sg{X}$.
Let $\GG_0$ be the abstract group defined as follows:

{\sc{Generators}}. $\GG_0$ is generated by the set of pairs $$(\omega,x).$$

{\sc{Relations}}. (1) For every redundant pair $(\omega,x)$, we have the relation $$(\omega,x) = 1.$$
(2) Next, for each $\omega \in \Omega$ and a relation $r \in R$ of $G$, write $r = x_{1} \cdots x_{n}$ for $x_{i} \in X$ (here we assume positivity). Define a sequence $\omega_i \in \Omega$ by $\omega_1 = \omega$ and $\omega_{i+1} = x_i^{-1}(\omega_i)$. Then
\begin{equation}\label{algo}
(\omega_1, x_1) (\omega_2, x_2) \cdots (\omega_n,x_n) = 1,
\end{equation}
is a defining relation of $\GG_0$ (called the ``interpretation of~$r$ at~$\omega$''); and these are all the relations.
\end{algo}

We thus provided a presentation for $\GG_0$ (which is finite if $G$ is finitely presented and $\Omega$ is finite).
\begin{prop}[Reidemeister-Schreier]\label{RMhere}
The map $\Psi_o \co \GG_0 \ra G$ defined by
\begin{equation}\label{Psidef}
\Psi_o \co (\omega,x) \mapsto \s_{\omega} x \s_{x^{-1}(\omega)}^{-1}
\end{equation}
is a well defined embedding, and its image is the stabilizer~$G_o$.
\end{prop}

\begin{cor}\label{thestabis}
As a subgroup of $G$, the stabilizer is
$$G_o = \sg{\s_\omega x \s_{x^{-1}\omega}^{-1} \suchthat \omega \in V(\Gamma), x \in E(\Gamma)}.$$
\end{cor}

\begin{exmpl}\label{DY5}
Consider the group $C(\triangle) = \sg{x,y,z}$ defined in \Eref{DY}. This group acts on the three vertices of the triangle. Take the base point $o = y \cap z$. Taking the Schreier set $\set{1, y, z}$, we find that $C(\triangle)_o = \sg{x, yzy, zxy}$.

Similarly, consider the same group acting on the four vertices of $\Y$. Take $o$ to be the center point of the graph. The Schreier set is $\set{1,x,y,z}$ and we find that $C(\Y)_o = \sg{xyx,yzy,zxz}$. (Compare to \Eref{DY3}.\eq{DY3.6} below.)
\end{exmpl}

\begin{rem}
In computational implementations it is customary to a priori remove the redundant generators.
\end{rem}

\begin{rem}
In the notation of \cite{MKS}, the generator $(\omega,x)$ is
$$s_{\omega,x} = \s_\omega x \overline{\s_\omega x}^{-1} = \s_\omega x \s_{x^{-1}\s_\omega^{-1}(o)}^{-1} = \s_\omega x \s_{x^{-1}(\omega)}^{-1},$$ and the relation is obtained from rewriting $$\s_{\omega} r \s_{\omega}^{-1} = (\s_{\omega_1} x_1 \s_{\omega_2}^{-1}) ( \s_{\omega_2} x_2 \s_{\omega_3}^{-1}) \cdots (\s_{\omega_n} x_n \s_{\omega_{n+1}}^{-1})$$
where $\omega_{n+1} = \omega_1$.
\end{rem}

\section{The stabilizer of a point}\label{sec5}

Let $\Gamma$ be any connected simple graph. As before, $C(\Gamma)$ acts transpositionally on the vertex set. In this section we point out a set of generators for the stabilizer. Choose a spanning subtree $\Sigma \sub \Gamma$. In particular $V(\Sigma) = V(\Gamma)$ and $E(\Sigma) \sub E(\Gamma)$. \Eq{pidef} defines a system of paths $\pi_{ab}$. Reading \Cref{thestabis} in this setup, we obtain:

\begin{cor}\label{workstoo}
Let $\Gamma$ be a connected simple graph with a spanning subtree $\Sigma$. Let $o \in V(\Gamma)$. Then
$$C(\Gamma)_o = \sg{\pi_{o\omega} x \pi_{o \omega}^{-1} \ (\omega \not \in x), \quad \pi_{oa} x \pi_{ob}^{-1} \ (x = \set{a,b})}.$$
\end{cor}

Following this corollary, we seek a concrete set of generators for the stabilizer. Let $\Gamma$ be a connected simple graph, and~$o$ a fixed vertex. The parabolic subgroup generated by the edges that are not adjacent to $o$ is clearly contained in the stabilizer, but this subgroup may be too small. For example, if $u,v$ meet at the vertex $o$, then the {\bf{bridge}} $uvu$ stabilizes $o$. See \Fref{bfb}.
Let $\set{u,v,w}$ be a fork centered at a vertex $\omega$, and assume $u_1,\dots,u_n$ is a path from $o$ to $\omega$, such that $u_n = u$. We call $u_1 \cdots u_n \cdot vwv \cdot u_n \cdots u_1$ a {\bf{pitchfork}} hitting $\set{u,v,w}$. This element stabilizes $o$. Similarly, let $w_1,\dots,w_m$ be a closed loop in the graph, and let $u_1,\dots,u_n$ be a path from $o$ to $w_1 \cap w_m$. We call $u_1 \cdots u_n \cdot w_1 \cdots w_m \cdot u_n \cdots u_1$ a {\bf{balloon}}, and notice that it too stabilizes~$o$.
\begin{figure}
$$\xymatrix{
*={\bullet} \ar@{-}[r] \ar@/^8pt/@{.}[rr] & o  \ar@{-}[r] & *={\bullet}
}
\qquad
\xy
    (8,0)*+{o};(16,0)  **\crv{\dir{*}} ?(1)*\dir{*};
    (16,0);(24,0)  **\crv{\dir{*}} ?(0)*\dir{*}?(1)*\dir{*};
    (24,0);(32,0)  **\crv{\dir{*}} ?(0)*\dir{*}?(1)*\dir{*};
    (32,0);(40,4)  **\crv{\dir{*}} ?(0)*\dir{*}?(1)*\dir{*};
    (32,0);(40,-4)  **\crv{\dir{*}} ?(0)*\dir{*}?(1)*\dir{*};
    (9,1);(9,-1)**\crv{~*{.}(10,2) & (32,2) & (34,4) & (44,8) & (44,-8) & (34,-4) & (32,-2) & (10,-2)} 
\endxy
\qquad
\xy
    (8,0)*+{o};(16,0)  **\crv{\dir{*}} ?(1)*\dir{*};
    (16,0);(24,0)  **\crv{\dir{*}} ?(0)*\dir{*}?(1)*\dir{*};
    (24,0);(32,0)  **\crv{\dir{*}} ?(0)*\dir{*}?(1)*\dir{*};
    (32,0);(40,6)  **\crv{\dir{*}} ?(0)*\dir{*}?(1)*\dir{*};
    (32,0);(40,-6)  **\crv{\dir{*}} ?(0)*\dir{*}?(1)*\dir{*};
    (40,6);(48,0)  **\crv{\dir{*}} ?(0)*\dir{*}?(1)*\dir{*};
    (40,-6);(48,0)  **\crv{\dir{*}} ?(0)*\dir{*}?(1)*\dir{*};
    (9,1);(9,-1)  **\crv{~*{.}(10,2) & (28,2) & (32,3) & (40,10) & (50,3) & (52,0) & (50,-3) & (40,-10) & (32,-3) & (28,-2) & (10,-2)}
\endxy
$$
\caption{A bridge, a pitchfork, and a balloon}\label{bfb}
\end{figure}
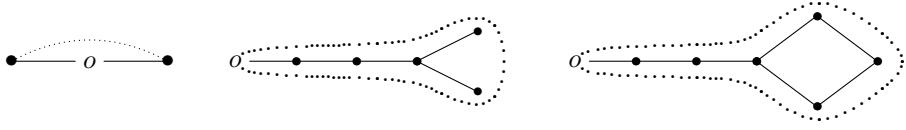

\begin{thm}\label{this5.2}
Let $\Gamma$ be a connected simple graph. Let $o \in V(\Gamma)$ be any vertex. The stabilizer $C(\Gamma)_o$ is generated by
\begin{enumerate}
\item the edges $x \in E(\Gamma)$ not adjacent to $o$;
\item the bridges over $o$;
\item the pitchforks based in $o$; and
\item the balloons based in $o$.
\end{enumerate}
\end{thm}
\begin{proof}
Let $\Sigma$ be a spanning subtree, and $\pi_{ab}$ the associated system of paths as in \Ssref{ss:pi}. Fix a vertex $o$. We compute the generators given in \Cref{workstoo}.

First let $\omega$ be any vertex, and $x$ an edge such that $\omega \not \in x$. Let $u_1,\dots,u_n$ be the path from $o$ to $\omega$ in the tree.
(i) If $x$ does not touch the path then $\pi_{o\omega}x\pi_{o\omega}^{-1} = x$.
(ii) If $x$ touches the path once, at $\omega' = u_i \cap u_{i+1}$, then $\pi_{o\omega}x\pi_{o\omega}^{-1}$ is a pitchfork hitting $\set{u_i,u_{i+1},x}$.
(iii) If $x$ touches the path once, at $o$, then $\pi_{o\omega}x\pi_{o\omega}^{-1} = u_1 x u_1$ is the bridge $u_1 x u_1$.
(iv) If $x = u_i$ on the path then $\pi_{o\omega}x\pi_{o\omega}^{-1} = u_{i+1}$.
(v) If $x$ touches the path twice, the outcome is as in case (ii) with $\omega'$ the farthest of the two.

Secondly, let $x = \set{a,b}$ be any edge. If $x \in \Sigma$ then $\pi_{oa}x\pi_{ob}^{-1} = 1$ as in \Rref{whotriv}. Otherwise, let $c$ be the departure point of the paths from $o$ to $a$ and $b$. Then $\pi_{oa} x \pi_{ob}^{-1} = \pi_{oc} \pi_{ca} x \pi_{bc} \pi_{co}$ is a balloon.
\end{proof}

The set of all pitchforks and balloons is unnecessarily large. The same proof leads to a fairly economic set of generators:
\begin{cor}
Let $\Gamma$ be a connected simple graph, with a spanning subtree $\Sigma$. Let $o \in V(\Gamma)$ be any vertex. The stabilizer $C(\Gamma)_o$ is generated by
\begin{enumerate}
\item  the edges $x \in E(\Gamma)$ not adjacent to $o$; \item the bridges $uvu$ over $o$ when at least one of the edges is in $\Sigma$;
\item  the pitchforks whose handle is from $\Sigma$ and at most one edge of the fork is not in $\Sigma$; and
\item  the balloons where the string is from $\Sigma$ and exactly one edge of the cycle is not from $\Sigma$.
\end{enumerate}
\end{cor}

In \Eref{DY5} we saw that $C(\Y)_o$, where $o$ is the center point, is generated by bridges. Here are a few other examples.

\begin{exmpl}\label{e55}
Let $\Gamma$ be the graph in \Fref{3inf}. Take $\Sigma = \Gamma - \set{z'}$ to be the spanning subtree. Take $o = 1$. Applying \Cref{workstoo}, we find that $C(\Gamma)_o$ is generated by the parabolic subgroup $\sg{v,x,y,z,y',z',x',p,q}$; the four pitchforks $uvxvu$, $uxyzyxu$, $uxyy'x'pqpx'y'yxu$, and $uxyy'x'z'x'y'yxu$; and the balloon $uxzz'y'yxu$ around the unique cycle.
\end{exmpl}

\begin{figure}
$$\xymatrix@R=6pt@C=10pt{
1 \ar@{-}[rd]^u & && 5 \ar@{-}[rd]^{y'} && & 9 \ar@{-}[ld]_p \\
{} & 3  \ar@{-}[r]^x \ar@{-}[ld]^v & 4 \ar@{-}[ru]^y \ar@{-}[rd]_z &  & 7 \ar@{-}[r]^{x'} & 8 \ar@{-}[rd]_q  \\
2 & && 6 \ar@{-}[ru]_{z'} && & 10
}
$$
\caption{The graph of \Eref{e55}} \label{3inf}
\end{figure}
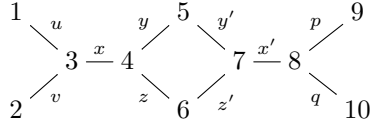

The union of paths joined in a central point is called a {\bf{generalized star}}.
\begin{exmpl}
Let $\Sigma$ be a generalized star. Denote the central vertex by $o$ and the edges adjacent to it by $v_1,\dots,v_k$. Then $C(\Sigma)_o$ is generated by the parabolic subgroup $\sg{E(\Sigma)-\set{v_1,\dots,v_k}}$ and the bridges $v_i v_j v_i$.
\end{exmpl}

\begin{exmpl}\label{Noa1}
Let $\Gamma = K_{n+1}$ be the complete graph on $n+1$ vertices. Denote the vertices by $0,1,\dots,n$, and the edges by $u_i = \set{0,i}$ and $v_{ij} = \set{i,j}$ for distinct $i,j \neq 0$. Then $C(\Gamma)_0$ is generated by the edges $v_{ij}$, the bridges $u_i u_j u_i$, and the balloons $u_i v_{ij} u_j$.
\end{exmpl}

\subsection{The tree case}\label{5*2}

The edges, bridges, and pitchforks occurring in \Tref{this5.2} are
reflections in $C(\Sigma)$. Since there are no balloons when
$\Sigma$ is a tree, the stabilizer $C(\Sigma)_o$ is generated by
reflections. By the theorems of Deodhar and Dyer \cite{Deodhar,Dyer},
it is therefore a Coxeter group.

\begin{cor}\label{here77}
Let $\Sigma$ be a tree. Then the point stabilizer $C(\Sigma)_o$ is a reflection subgroup of $C(\Sigma)$, and hence is itself a Coxeter group.
\end{cor}
\Sref{sec:6} is devoted to identifying the Coxeter diagram of the stabilizer.

\section{A presentation for the stabilizer}\label{sec:pres}

We now restart the previous section and aim higher, namely for a presentation of the stabilizer. Let $\Gamma$ be any connected simple graph, and $C(\Gamma)$ the associated Coxeter group. Our goal is to find a presentation for the stabilizer $C(\Gamma)_o$, where $o$ is a vertex of $\Gamma$. Choose a spanning subtree $\Sigma \sub \Gamma$. In particular $V(\Sigma) = V(\Gamma)$ and $E(\Sigma) \sub E(\Gamma)$.

To obtain a presentation for the stabilizer $C(\Gamma)_o$, we apply \Aref{RSalg} to the group $G = C(\Gamma)$ acting on the vertex set. Using \eq{pidef} for paths in the spanning tree~$\Sigma$, we take $\s_{a} = \pi_{oa}$. The set of such $\s_{a}$ ($a \in V(\Sigma)$) is a Schreier set because the initial subword $u_1 \dots u_{t'}$ of $\s_a = u_1 \cdots u_t$ is~$\s_{a'}$ for $a' = (u_1 \cdots u_{t'})^{-1}o$. Since the presentation depends on the choice of the Schreier set, which for us is determined by the choice of the spanning subtree, we denote the abstract group defined by the algorithm by~$\GG_0(\Sigma,\Gamma)$. By Reidemeister-Schreier (\Pref{RMhere}) we know that $C(\Gamma)_o \isom \GG_0(\Sigma,\Gamma)$ for any spanning subtree $\Sigma$.

By definition, $\GG_0(\Sigma,\Gamma)$ is generated by the pairs $(a,x)$ where $a \in V(\Sigma)$ and $x \in E(\Gamma)$, but some generators are redundant.
\begin{rem}\label{whotriv}
For any edge $x$ of the spanning tree $\Sigma$, if $a$ is a vertex on $x$ then $(a,x)$ is redundant or trivial.

Indeed, by assumption $xa$ is a neighbor of $a$ in the tree. If $\dist[\Sigma]{o}{xa}> \dist[\Sigma]{o}{a}$ then $(a,x)$ is redundant because $\s_{xa} = \s_a x$ by definition. On the other hand if $\dist[\Sigma]{o}{xa} < \dist[\Sigma]{o}{a}$ then the interpretation of the relation $x^2 = 1$ at~$a$ is $(a,x)(xa,x) = 1$, but since $(xa,x)$ is redundant we get $(a,x) = 1$.
\end{rem}
In light of this remark, and to simplify notation, we will include redundant generators in the presentation, together with the relation defining them to be trivial.
Now $\GG_0(\Sigma,\Gamma)$ is generated by all the pairs $(a,x)$ (for all $a \in V(\Sigma) = V(\Gamma)$ and $x \in E(\Gamma)$), and $(a,x) = 1$ if $a \in x \in E(\Sigma)$.

\begin{prop}\label{52}
The generators of $\GG_0(\Sigma,\Gamma)$ satisfy the following identities, for any $x \in E(\Gamma)$:
\begin{enumerate}
\item If $xa = a$ then $(a,x)^2 = 1$.
\item If $xa \neq a$ then $(xa,x) = (a,x)^{-1}$. If, furthermore, $x \in E(\Sigma)$, then $(xa,x)=(a,x)=1$.
\end{enumerate}
\end{prop}
Consequently, we say that the generator $(a,x)$ is of {\bf{type $2$}} if $xa = a$ and {\bf{type~$\infty$}} if $xa \neq a$, reflecting the order of the generator. We will prove below that generators of type $\infty$, for edges not in $\Sigma$, are indeed of infinite order.
\begin{proof}[Proof of \Pref{52}]
Both claims follow by interpreting the relation $x^2 = 1$ at various vertices.
First assume that $xa = a$, namely $a$ is not a vertex on $x$. The interpretation by \eq{algo} is $(a,x)(a,x)=1$, which proves the first claim. Similarly if~$a$ is a vertex on $x$, then $b = xa$ is the other vertex, and the interpretation at $a$ is $(a,x)(b,x)=1$, which shows that $(b,x) = (a,x)^{-1}$. As remarked above, if in this case $x \in E(\Sigma)$ then $(a,x) = (b,x) = 1$.
\end{proof}

To distinguish the types, we will henceforth write $x_a$ for the generators $(a,x)$ of type $\infty$. \begin{rem}\label{whatisPsi}
After establishing the notation for the generators of $\GG_0(\Sigma,\Gamma)$, let us make explicit the injection $\Psi_o \co \GG_0(\Sigma,\Gamma) \ra C(\Gamma)$, whose image is the stabilizer~$C(\Gamma)_o$. By \eq{Psidef} and our choice of the $\s_a$, this map is defined, for any $x \in E(\Gamma)$, by
\begin{eqnarray}
  (\omega,x) & \mapsto & \pi_{o \omega} x \pi_{o \omega}^{-1}  \qquad \omega \not \in x
\label{Psidefhere1} \\
  x_a & \mapsto & \pi_{oa} x \pi_{ob}^{-1} \qquad x = \set{a,b}
\label{Psidefhere2}
\end{eqnarray}
\end{rem}

Applying Algorithm~\ref{RSalg} to the presentation given in \Dref{CT} results in the following.
\begin{prop}\label{thepres}
The group $\GG_0(\Sigma,\Gamma)$ is generated by two types of generators:
\begin{enumerate}
\item $(\omega,x)$ for any edge $x \in E(\Gamma)$ and any vertex~$\omega$ not on $x$;
\item $x_a$        for any edge $x \in E(\Gamma)$ and any vertex $a \in x$;
\end{enumerate}
subject to the following relations:
\begin{enumerate}
\item For every edge $x = \set{a,b}$ of $\Gamma$: \label{54.1}
\begin{enumerate}
\item $(\omega,x)^2 = 1$ \quad for every $\omega \neq a,b$;
\item $x_ax_b = 1$;
\item\label{54.1c} $x_a = x_b = 1$ \quad when $x \in E(\Sigma)$;
\end{enumerate}
\item For every pair $x = \set{a,b}$ and $y = \set{c,d}$ of disjoint edges in $\Gamma$:
\begin{enumerate}
\item\label{54.2a} $[(\omega,x),(\omega,y)] = 1$ \quad for every $\omega \neq a,b,c,d$;
\item\label{54.2b} $x_a(b,y)x_a^{-1} = (a,y)$;
\item\label{54.2c} $y_c(d,x)y_c^{-1} = (c,x)$;
\end{enumerate}
\item For every pair $x = \set{a,b}$ and $y = \set{b,c}$ of (distinct) intersecting edges in~$\Gamma$:
\begin{enumerate}
\item\label{54.3a} $(\omega,x)(\omega,y)(\omega,x) = (\omega,y)(\omega,x)(\omega,y)$ \quad for every $\omega \neq a,b,c$;
\item\label{54.3b} $y_b(c,x)y_b^{-1} = x_b(a,y)x_b^{-1}$. 
\end{enumerate}
\end{enumerate}
\end{prop}
\begin{proof}
The relations are obtained by interpreting $xx = 1$, $xyxy=1$ and $xyxyxy = 1$ at every vertex, as in \Eq{algo}. {}From $xx=1$ we get $(\omega,x)(x \omega, x) =1$, which, taking into account that $(a,x)=1$ when $x \in E(\Sigma)$, is the part \eq{54.1}. Similarly when $x=\set{a,b}$ and $y = \set{c,d}$, $xyxy=1$ gives $(\omega,x)(x \omega ,y)(yx\omega,x)(xyx\omega,y) =1$, which reduces to \eq{54.2a} if $\omega \neq a,b,c,d$; to $(a,x)(b,y)(b,x)(a,y)=1$ when $\omega = a$, and so on. Finally for $x = \set{a,b}$ and $y = \set{b,c}$, $xyxyxy=1$ gives $$(\omega,x)(x\omega,y)(yx\omega,x)(xyx\omega,y)(yxyx\omega,x)(xyxyx\omega,y) = 1,$$
which is \eq{54.3a} if $\omega \neq a,b,c$; and to
                    $(a,x)(b,y)(c,x)(c,y)(b,x)(a,y) = 1$ for $\omega = a$,
          $(b,x)(a,y)(a,x)(b,y)(c,x)(c,y) = 1$ for $\omega = b$, and
$(c,x)(c,y)(b,x)(a,y)(a,x)(b,y) = 1$ for $\omega = c$ -- which are all rotations of \eq{54.3b}.
\end{proof}

The fundamental group $\pi_1(\Gamma)$ of the graph $\Gamma$ is the free group of rank $r$, where $r = \card{E(\Gamma)-E(\Sigma)}$ is the number of cycles, namely the number of excess edges over the spanning tree.
\begin{cor}\label{free}
There is a projection $\GG_0(\Sigma,\Gamma) \ra \pi_1(\Gamma)$.
\end{cor}
\begin{proof}
Mod out by the normal subgroup generated by all the elements $(\omega,x)$ for $\omega \not \in x$. The quotient is generated by the $x_a$, one generator for each $x \in E(\Gamma) - E(\Sigma)$, and all the relations vanish. The resulting group is indeed the free group $\pi_1(\Gamma)$.
\end{proof}
A concrete map $C(\cyc)_o \ra \Z$ is given in \Eref{cyccase}.

\begin{cor}\label{notCoxeter}
If $\Gamma$ is not a tree, then $C(\Gamma)_o$ is not a Coxeter group. (Because a Coxeter group has no nontrivial torsion-free quotient).
\end{cor}

\begin{cor}
The subgroup generated by all the $x_a$ in $\GG_0(\Sigma,\Gamma)$ (one for each edge not in $\Sigma$) is free. In particular the order of each $x_a$ is infinite.
\end{cor}

From the presentation of $\GG_0(\Sigma,\Gamma)$ we can describe the abelianization of $\GG_0(\Sigma,\Gamma)$ in terms of the fundamental group of $\Gamma$.
\begin{thm}\label{abt}
The abelianization of
$\GG_0(\Sigma,\Gamma)$, when $\card{V(\Gamma)}>2$, is
$$\ab{\GG_0(\Sigma,\Gamma)} \isom \Z_2 \,\times\, \ab{\pi_1(\Gamma)};$$
the first component is the image of all the $(\omega,x)$, and the second is generated by the images of the $x_a$.
\end{thm}
\begin{proof}
Reading through the presentation of $\GG_0(\Sigma,\Gamma)$ given in \Pref{thepres}, the abelianization is the abelian group generated by the $(\omega,x)$ and $x_a$ as above, subject to the relations:
\begin{enumerate}
\item\label{AB0}
$x_ax_b = 1$ \quad for every edge $x = \set{a,b}$ not in $\Sigma$;
\item\label{AB1} $x_a = 1$ \quad for every edge $x = \set{a,b}$ in $\Sigma$;
\item $(\omega,x)^2 = 1$ \quad for every edge~$\omega$ not on $x$;
\item\label{AB2} $(b,y) = (a,y)$ \quad for any disjoint edges $x = \set{a,b}$ and $y = \set{c,d}$ in $\Gamma$;
\item For every pair $x = \set{a,b}$ and $y = \set{b,c}$ of intersecting edges in~$\Gamma$:
\begin{enumerate}
\item\label{AB3.1} $(\omega,x) = (\omega,y)$ \quad for every $\omega \neq a,b,c$;
\item\label{AB3.2} $(c,x) = (a,y)$;
\end{enumerate}
\end{enumerate}

Since the only relations involving the generators $x_a$ are \eq{AB0} and \eq{AB1}, we have that
$$\ab{\GG_0(\Sigma,\Gamma)} \isom \sg{(\omega,x) \suchthat x \in E(\Gamma), \omega \not \in x} \times \sg{x_a \suchthat x \in E(\Gamma), a \in x},$$
and that the second component is generated by the $x_a$ for $x \not \in E(\Sigma)$, subject only to the relation $x_{xa} = x_a^{-1}$. This component is therefore isomorphic to $$\Z^r \isom \ab{\pi_1(\Gamma)}.$$

It remains to show that all the generators $(\omega,x)$ coincide in the abelianization (noting that such generators exist because of the assumption on the number of vertices in $\Gamma$). To prove this, we prove two subclaims. First we show that
$(\omega,x) = (\omega,y)$ for any two edges $x,y$, and any vertex~$\omega$ not on $x,y$. Let $x = u_1,\dots,u_t = y$ be a simple path starting at $x$ and ending at $y$. If~$\omega$ is not a vertex on the path, we are done by repeated application of Relation \eq{AB3.1}. If~$\omega$ is on the path, then $t \geq 4$, and we may still apply the same relation and shorten the path, unless $t = 4$, where (labelling the vertices) the path is
$$\xymatrix{a \ar@{-}[r]^{x} & b \ar@{-}[r]^{u_2} & \omega \ar@{-}[r]^{u_3} &c \ar@{-}[r]^{y} & d}.$$
Applying the relations \eq{AB2}, \eq{AB3.1} and \eq{AB3.2} as indicated, we get
$$(\omega,x) \stackrel{\eq{AB2}}{\equiv} (c,x) \stackrel{\eq{AB3.1}}{\equiv} (c,u_2) \stackrel{\eq{AB3.2}}{\equiv} (b,u_3) \stackrel{\eq{AB3.1}}{\equiv} (b,y) \stackrel{\eq{AB2}}{\equiv} (\omega,y).$$

Secondly, we show that for any edge $x$, $(\omega,x) = (\omega',x)$ for any two vertices $\omega,\omega'$ not on $x$. Let $u_1,u_2,\dots,u_n$ be a shortest path from~$\omega$ to $\omega'$. If the path does not touch $x$, repeatedly apply Relation~\eq{AB2}. If $x$ is one of the edges on the path, then by the same relation we may assume the path is
$$\xymatrix{\omega \ar@{-}[r]^{u_1} & a \ar@{-}[r]^{x} & b \ar@{-}[r]^{u_3} & \omega'},$$
where $(\omega,x) \stackrel{\eq{AB3.1}}{\equiv} (\omega,u_3) \stackrel{\eq{AB2}}{\equiv} (a,u_3) \stackrel{\eq{AB3.2}}{\equiv} (\omega',x)$. Finally if the path touches $x$ in one vertex, again we may shorten the path by Relation~\eq{AB2}, arriving at the claw
$$\xymatrix@C=18pt@R=10pt{
 {} & \omega \ar@{-}[r]^{u_1} & a \ar@{-}[d]^x & \omega' \ar@{-}[l]_{u_2} \\
 {} & {} & b & {}},$$
where $(\omega,x) \stackrel{\eq{AB3.1}}{\equiv} (\omega,u_2) \stackrel{\eq{AB3.2}}{\equiv} (\omega',u_1) \stackrel{\eq{AB3.1}}{\equiv} (\omega',x)$. With this, all the relations $(\omega,x)$ collapse to a single element of order $2$, generating the first component of the abelianization. 
\end{proof}

\begin{rem}
The proof of \Tref{abt} also shows that in the group $\GG_0(\Sigma,\Gamma)$, all the $(\omega,x)$ are conjugate.
\end{rem}

In conclusion of this section, we pose the following natural question. The key property of $C(\Gamma)$ is that it acts on the vertex set $V(\Gamma)$; letting $n = \card{V(\Gamma)}$ be the number of vertices, we get a short exact sequence
\begin{equation}\label{basics}
\xymatrix{1 \ar@{->}[r] & \bigcap_{o \in V(\Gamma)} C(\Gamma)_o \ar@{->}[r] & C(\Gamma) \ar@{->}[r] & S_n\ar@{->}[r] & 1,}
\end{equation}
where each $C(\Gamma)_o$ is the stabilizer of a vertex $o$, and the intersection is taken over all the vertices.

\begin{ques}\label{kernel}
Let $\Gamma$ be a connected simple graph with $n = \card{V(\Gamma)}$. Find a presentation for the kernel of $C(\Gamma) \ra S_n$.
\end{ques}
A concise description of the kernel will be quite helpful for the computation of fundamental groups of Galois covers, see \cite{Merav8}. Note that the kernel of $\CY(\Gamma) \ra S_n$ was fully described in \cite{Cox}.

\section{The stabilizer $C(\Sigma)_o$ for a tree}\label{sec:6}

We continue the previous section, assuming here that $\Gamma = \Sigma$ is a tree.
After showing that the stabilizer in this case is a Coxeter group in \Ssref{5*2}, our goal here is to identify the corresponding Coxeter diagram as~$\covg{\Sigma}$.

Again let $\GG_0(\Sigma) = \GG_0(\Sigma,\Sigma)$ denote the abstract group isomorphic
to the stabilizer $C(\Sigma)_o$ by the isomorphism $\Psi_o \co \GG_0(\Sigma) \ra C(\Sigma)_o$ defined in \eq{Psidef}, where $o$ is a fixed vertex. Simplifying the presentation from \Pref{thepres}, we obtain:
\begin{cor}\label{relstree}
Suppose $\Gamma = \Sigma$ is a tree. The group $\GG_0(\Sigma)$ is
generated by the pairs $(\omega,x)$ for $x \in E(\Sigma)$ and~$\omega$ a vertex not on $x$; with the following relations:
\begin{enumerate}
\item For every edge $x = \set{a,b}$ of $\Sigma$:
\begin{enumerate}
\item\label{T1a} $(\omega,x)^2 = 1$ for every $\omega \neq a,b$;
\end{enumerate}
\item For every pair $x = \set{a,b}$ and $y = \set{c,d}$ of disjoint edges in $\Sigma$:
\begin{enumerate}
\item\label{T2a} $[(\omega,x),(\omega,y)] = 1$ for every $\omega \neq a,b,c,d$;
\item\label{T2b} $(b,y) = (a,y)$;
\item\label{T2c} $(d,x) = (c,x)$;
\end{enumerate}
\item For every pair $x = \set{a,b}$ and $y = \set{b,c}$ of (distinct) intersecting edges in~$\Gamma$:
\begin{enumerate}
\item\label{T3a} $(\omega,x)(\omega,y)(\omega,x) = (\omega,y)(\omega,x)(\omega,y)$ for every $\omega \neq a,b,c$;
\item\label{T3b} $(a,y) = (c,x)$;
\end{enumerate}
\end{enumerate}
\end{cor}

\begin{figure}
$$\xymatrix{
{} & \GG_0(\Sigma) \ar@{.>}[dl]_(0.45){\Theta} \ar@{->}[dr]^{\Psi_o} & {} \\
\Cover \ar@{->}[rr]^{\Phi_o} & & C(\Sigma)_o
}$$
\caption{The stabilizer in three disguises}\label{thegroups2}
\end{figure}
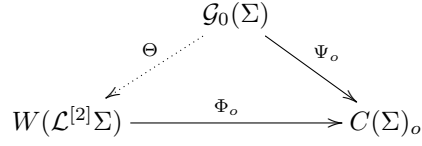

In \Pref{map} we constructed maps $\Phi_o \co \Cover \ra C(\Sigma)$, whose images are contained in the point stabilizers $C(\Sigma)_o$. 
Our goal is to prove that the three groups in \Fref{thegroups2}---one given by presentation, one a Coxeter group, and one the point stabilizer---are isomorphic. In particular, this will give $C(\Sigma)_o$ a concrete presentation as the Coxeter group $\Cover$, which was defined in \Dref{covgdef}. The next step is to define a homomorphism $\Theta \co \GG_0(\Sigma) \ra \Cover$. Let $(\omega, x)$ be a generator. Thus $x$ is an edge and~$\omega$ a vertex not on $x$. Write $x = \set{a,b}$ where~$\omega$ is closer to $b$ than to $a$. Let $c$ be the neighbor of $b$ on the path from~$b$ to~$\omega$, as in~\eq{patho}. We define $\Theta$ to send $(\omega,x)$ to the edge $\set{x,x'} = [abc]$ of $\Line{\Sigma}$.
\begin{equation}\label{patho}
\xymatrix{
a \ar@{-}[r]^{x} & b \ar@{-}[r]^{x'} & c \ar@{. }[rr] &  & \omega}
\end{equation}

\begin{prop}\label{Thetadef}
The map $$\Theta \co \GG_0(\Sigma) \ra \Cover$$
defined by sending each generator $(\omega,x)$ to $[abc] \in E(\Line{\Sigma})$ as above is a well-defined homomorphism.
\end{prop}
\begin{proof}
We apply $\Theta$ to the relations listed in \Cref{relstree}. The relations in \ref{relstree}.\eq{T1a} become $[abc]^2 = 1$, which clearly holds in $\Cover$. 

To verify \ref{relstree}.\eq{T2a}, let $x=\set{a,b}$ and $y = \set{c,d}$ be disjoint edges in $\Sigma$, and let~$\omega$ be a vertex not on these edges. Let $x',y'$ be the edges for which~$\Theta$ maps $(\omega,x)$ and $(\omega,y)$ to $\set{x,x'} = [abb']$ and $\set{y,y'} = [cdd']$, respectively (thus $\dist{b'}{\omega} < \dist{b}{\omega} < \dist{a}{\omega}$ and
$\dist{d'}{\omega} < \dist{d}{\omega} < \dist{c}{\omega}$). To show that $\set{x,x'}$ and $\set{y,y'}$ are disjoint, it suffices to verify that $x' \neq y'$; but otherwise $b' = d$ and $d' = b$ (since $b \neq d$), which is impossible because $\dist{b'}{\omega} < \dist{b}{\omega}$ and $\dist{d'}{\omega} < \dist{d}{\omega}$. Furthermore, since~$\omega$ is closer to $b',d'$ than to $a,d$, respectively, the pairs $\set{x,x'}$ and $\set{y,y'}$ are not evenly connected. This proves that $\set{x,x'}$ and $\set{y,y'}$ commute by the definition of $\Cover$. 

Consider the relation \ref{relstree}.\eq{T2b}. Assuming $a$ is closer to $\set{c,d}$ than $b$ is, the path from $b$ to $\set{c,d}$ passes through $a$, so they approach $\set{c,d}$ together. This shows that~$\Theta$ maps $(a,y)$ and $(b,y)$ to the same edge of $E(\Line{\Sigma})$, so the relation holds trivially in the image. Similarly for \ref{relstree}.\eq{T3b}, when $x = \set{a,b}$ and $y = \set{b,c}$, both $(a,y)$ and $(c,x)$ map to $[abc]$.

It remains to verify Relation \ref{relstree}.\eq{T3a}. Let $\omega \not \in \set{a,b,c}$, where $x = \set{a,b}$ and $y = \set{b,c}$ are edges. If~$\omega$ is closer to $a$ than to~$b$ and $c$, then $(\omega,x) \mapsto \set{x',x}$ and $(\omega,y) \mapsto \set{x,y}$, where $x'$ is the edge leading from $a$ towards~$\omega$. The product of two edges of $\Line{\Sigma}$ is of order $3$ by the definition of $\Cover$. 
The case where $\omega$ is closer to $c$ than to $b$ and $a$ is symmetric.
Finally if~$\omega$ is closer to~$b$ than to $a$ or $c$, then $(\omega,x) \mapsto \set{x,z}$ and $(\omega,y) \mapsto \set{y,z}$ where $z$ is an edge from~$b$ towards~$\omega$, and the same holds.
\end{proof}

\begin{prop}
The triangle in \Fref{thegroups2} commutes.
\end{prop}
\begin{proof}
We need to prove that $\Phi_o \circ \Theta = \Psi_o$. Let $(\omega,x)$ be a generator of $\GG_0$. Suppose $\Theta \co (\omega,x) \mapsto [abc]$, where $x = \set{a,b}$. As defined in \eq{Phidef}, $\Phi_o([abc]) = \pi_{oc}\pi_{ab}\pi_{oc}^{-1}$. At the same time, $\Psi_o$ maps $(\omega,x)$ to $\s_\omega x \sigma_{x\omega}^{-1} = \pi_{o \omega} x \pi_{o \omega}^{-1} = \pi_{oc}\pi_{c \omega} x \pi_{c \omega}^{-1} \pi_{o c}^{-1} = \pi_{oc}x  \pi_{o c}^{-1} = \pi_{oc} \pi_{ab} \pi_{oc}^{-1}$ by \Pref{pitrans} and the fact that the path from~$\omega$ to $c$ does not touch $x = \set{a,b}$ by the very choice of $c$, as in \eq{patho}.
\end{proof}

\begin{cor}
$\Theta$ and $\Phi_o$ are isomorphisms.
\end{cor}
\begin{proof}
Since $\Psi_o = \Phi_o \circ \Theta$ is an isomorphism, we conclude that $\Theta$ is injective. But also notice that $\Theta$ is surjective because every edge $[abc] \in E(\Line{\Sigma})$ is covered by $(a,\set{bc})$ and $(c,\set{a,b})$. This proves the claim for $\Theta$, and then for $\Phi_o = \Psi_o \circ \Theta^{-1}$ as well.
\end{proof}

This finally gives:
\begin{thm}\label{main-1}
$C(\Sigma)_o$ is isomorphic to the Coxeter group $\Cover$.
\end{thm}

Combining with \Cref{notCoxeter}, we proved:
\begin{cor}
For a graph $\Gamma$, the stabilizer $C(\Gamma)_o$ is a Coxeter group if and only if $\Gamma$ is a tree.
\end{cor}

\subsection{Rank comparison}

The Coxeter rank of $C(\Sigma)$ is, clearly, $\card{E(\Sigma)} = \card{V(\Sigma)} - 1$. By \Tref{main-1}, the rank of the stabilizer is the number of generators of $\Cover$, which is the number of vertices in $\covg{\Sigma}$. But this graph has the same vertex set as~$\Line{\Line{\Sigma}}$, composed of the edges $[abc]$ of $\Line{\Sigma}$, where $b$ is a vertex and $a,c$ distinct neighbors.

Let $\delta(b)$ denote the number of neighbors of a vertex $b$. The rank of $C(\Sigma)_o$ is, therefore
$$\rank C(\Sigma)_o = \sum_{v} \binom{\delta(v)}{2},$$
summing over the vertices $v \in V(\Sigma)$. Let us compare the ranks of $C(\Sigma)$ and the stabilizer.

For a tree $\Sigma$, define the invariant $i(\Sigma) = \sum \binom{\delta(v)}{2} - \card{E(\Sigma)}$. Let $n_d$ denote the number of vertices with $\delta(v) = d$, so that $i(\Sigma) = \sum n_d \binom{d}{2} - (\sum n_d-1)$.
The handshake lemma $\sum d n_d = 2 \card{E(\Sigma)}$ gives us $n_1 = 2+\sum_{d \geq 3}(d-2)n_d$. Consequently, $i(\Sigma) = 1-n_1+\sum [\binom{d}{2}-1]n_d = \sum_{d \geq 3} \binom{d-1}{2}n_d - 1$. The only case where $i(\Sigma)<0$ is when $\Sigma$ is path, in which case $C(\Sigma)$ is finite. Whenever $C(\Sigma)$ is infinite, we proved that
$$\rank C(\Sigma)_o \geq \rank C(\Sigma),$$
as $i(\Sigma) = \rank C(\Sigma)_o - \rank C(\Sigma) \geq 0$, in agreement with Felikson–Tumarkin \cite{Felikson1}. Equality holds only when $\Sigma$ is a generalized $\Y$.

\subsection{Coxeter pairs and examples}

In more common terms, since $C(\Sigma) = W(\Line{\Sigma})$  by \Cref{C=WL}, \Tref{main-1} provides the following embedding.
\begin{thm}\label{nice}
For any tree $\Sigma$, there is an embedding of Coxeter groups
$$\Cover \hra W(\Line{\Sigma}),$$
whose image is a reflection subgroup of index $\card{V(\Sigma)}$.
\end{thm}
After we proved in \Cref{here77} that the stabilizer is a reflection group, this theorem identifies its Coxeter diagram. This is the orbit-stabilizer count for the natural action of $C(\Sigma) = W(\Line{\Sigma})$ on~$V(\Sigma)$. Fixing any vertex $o \in V(\Sigma)$ and recalling \Eq{pidef}, an explicit embedding $\Phi_o \co \Cover \ra W(\Line{\Sigma})$ is given in \Eq{Phidef} by $[abc] \mapsto \pi_{oa}\pi_{bc}\pi_{oa}^{-1}$.

\begin{exmpl}\label{star4}
For $\Gamma = K_{1,n}$ a star with $n$ leaves, the line graph $\Line{\Gamma} = K_n$ is the complete graph. Taking $o$ to be the central vertex, $C(\Gamma)_o$ is generated by $u_{ij}$ for $\set{i,j} \sub \set{1,\dots,n}$ ($i\neq j$) where $u_{ij}=u_{ji}$, with the relations $u_{ij}^2 = 1$ and $(u_{ij}u_{ik})^3 = 1$ (distinct $i,j,k$). 
When $\set{i,j}$ and $\set{k,\ell}$ are disjoint, the edge connecting these vertices in $\covg{\Gamma}$ carries the label~$\infty$.
\end{exmpl}

\begin{exmpl}\label{e37+}
\Fref{2inf+} adds labels to the graphs of \Fref{2inf} (see \Eref{e37}). Let $\Sigma$ denote the left-most diagram; observe that $\Line{\Sigma}$ is the middle diagram, and~$\covg{\Sigma}$ is on the right. Labelling the edges as $u = \set{1,3}$, $x=\set{3,4}$, and so on, as shown in the diagram, the vertices of $\covg{\Sigma}$ are denoted as $[134] = [ux]$, etc.

Taking $o = 8$ as the fixed point, the Schreier set is $$\set{\pi_{81},\cdots,\pi_{88}} = \set{qzyxu, qzyxv, qzyx, qzy, qz, q, qp, 1}.$$
Let us compute the map $\Phi_o \co \Cover \hra W(\Line{\Sigma})$.
The eight generators of $\Cover$ are mapped to $u,v,x,y,z,p$ which clearly stabilize $8$, and to the pitchforks $\Phi_8([uv]) = \Phi_8([132]) = \pi_{82} \pi_{13} \pi_{28} = qzyxvuvxyzq$ and
$\Phi_8([zp]) = \Phi_8([567]) = \pi_{85} \pi_{67} \pi_{58} = qzpzq$.

Other stabilizers can be obtained by conjugation. Conjugating by $\pi_{58} = zq$ gives $C(\Sigma)_5 = \sg{yxuvuxy, u, v, x, zyz, q, zqpqz, z}$. Here $\Phi_5([yz]) = yzy$ bridges the connected parabolic subgroups $\sg{u,v,x}$ and $\sg{q,z}$; and $\Phi_5([pq]) = zpqpz$ and $\Phi_5([uv]) = yxuvuxy$ are the pitchforks based in~$5$.
\end{exmpl}
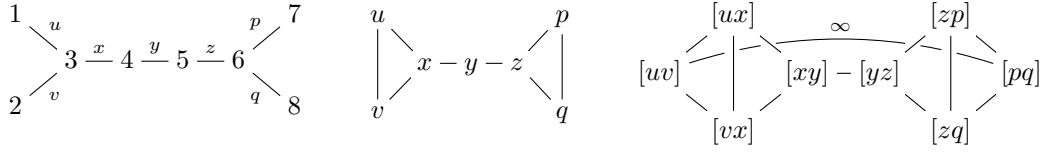
\begin{figure}
$$\xymatrix@R=5pt@C=10pt{
1 \ar@{-}[rd]^u & &&& & 7 \ar@{-}[ld]_p \\
{} & 3  \ar@{-}[r]^x \ar@{-}[ld]^v & 4 \ar@{-}[r]^y & 5 \ar@{-}[r]^z & 6 \ar@{-}[rd]_q  \\
2 & &&& & 8
}
\qquad
\xymatrix@R=6pt@C=6pt{
u \ar@{-}[dd] \ar@{-}[rd] &&& & p \ar@{-}[dd] \ar@{-}[ld]\\
{} & x  \ar@{-}[ld] \ar@{-}[r] & y \ar@{-}[r] & z \ar@{-}[rd] & \\
v && && q
}
\qquad
\xymatrix@R=6pt@C=5pt{
{} & {}[ux] \ar@{-}[rd] \ar@{-}[dd] & {} && [zp]  \ar@{-}[dd] & {} \\
{}[uv] \ar@/^3ex/@{-}[rrrrr]|(0.2)\hole|(0.3)\hole|(0.7)\hole|(0.8)\hole^\infty \ar@{-}[ru] \ar@{-}[rd] & {} & [xy] \ar@{-}[r] & [yz]  \ar@{-}[ru] \ar@{-}[rd] && {}[pq] \ar@{-}[lu] \ar@{-}[ld] \\
{} & {}[vx]  \ar@{-}[ru] & {} && [zq] & \\
}
$$
\caption{The graphs of \Eref{e37+}} \label{2inf+}
\end{figure}

\begin{exmpl}\label{DY3}
Let $\Gamma = \Y$ be the claw of \Fref{basic}, with the vertices numbered $0$ for the center, and $1,2,3$ for the leaves of $x,y,z$, respectively. The Coxeter group $C(\Y)$ acts on $\set{0,1,2,3}$ by $x \mapsto (01)$, $y \mapsto (02)$ and $z \mapsto (03)$. The spanning subtree is $\Sigma = \Gamma$. Thus $\pi_{01} = x$, $\pi_{02} = y$ and $\pi_{03} = z$. We choose the fixed vertex to be $o = 0$.
\begin{enumerate}
\item The group $\Cover[\Y]$ is generated by the edges $[102], [103], [203]$ of $\Line{\Y}$, with the relations $\lam^2 = 1$ and $(\lam\lam')^3 = 1$ for every distinct generators $\lam,\lam'$.
\item
Then
$\GG_0(\Y)$ is generated by
$(1,y)=(2,x)$, $(2,z)=(3,y)$, and $(3,x)=(1,z)$, with the
relations $\lam^2 = 1$ and $(\lam\lam')^3 = 1$ for every distinct generators $\lam,\lam'$.
\item The isomorphism $\Theta \co \GG_0(\Y) \ra \Cover[\Y]$ is defined by $(3,x) \mapsto [103]$, $(2,z) \mapsto [302]$, and $(1,y) \mapsto [201]$.
\item The group $\GG_0(\Y)$ is isomorphic to the stabilizer of $0$ under the action of $C(\Y)$. It can be identified as the image of $\Phi_0 \co \Cover \ra C(\Y)$, defined in \eq{Phidef}. Taking $b = 0$ we get $[a0c] \mapsto \pi_{0a}\pi_{0c}\pi_{0a}^{-1}$; namely $[103] \mapsto xzx$, $[201] \mapsto yxy$, and $[302] \mapsto zyz$. It follows that $C(\Y)_0 = \sg{xzx, yxy, zyz}$.
\item\label{DY3.emb} We thus obtained an embedding $$W(\triangle) \isom \Cover[\Y] \isom C(\Y)_0 \hra C(\Y) = W(\Line{\Y}) = W(\triangle),$$ of index $4$, defined by sending $x,y,z$ to $xzx,yxy,zyz$. 
\item\label{DY3.6} Conjugating $C(\Y)_0$ by $z$ takes this subgroup to $C(\Y)_3 = \sg{x, y, zyxyz} = \sg{x,y,\alpha^2,\beta^2,\gamma^2}$; this can be extracted directly from the \Eref{DY2}, which gives the kernel of $C(\Y) \ra S_4$.
\end{enumerate}
\end{exmpl}

\begin{exmpl}\label{ex69}
Let $n,n',n'' \geq 1$.
Let $\Y_{n,n',n''}$ denote the generalized $\Y$ tree, obtained by gluing the paths $\path[n]$, $\path[n']$, and $\path[n'']$
in a common vertex. Let $\triangle_{n,n',n''} = \Line{\Y_{n+1,n'+1,n''+1}}$, which is the triangle graph to which paths
$\path[n]$, $\path[n']$, and $\path[n'']$ are glued.
For example $\Y = \Y_{2,2,2}$ and $\triangle = \triangle_{1,1,1}$. Now let $\triangle^*_{n,n',n''} = \Line{\triangle_{n+1,n'+1,n''+1}}$.  This is the triangle graph, subdivided into four triangles, with paths of lengths $n, n',n''$ glued at the external vertices. The graphs are illustrated in \Fref{Deltastar}.

As noted in \Rref{YY+},
$$\covg{\Y_{n+2,n'+2,n''+2}} = \Line{\Line{\Y_{n+2,n'+2,n''+2}}} = \Line{\triangle_{n+1,n'+1,n''+1}} = \triangle^*_{n,n',n''}.$$
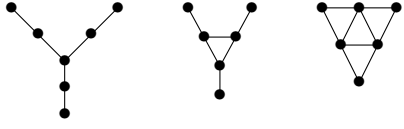
\begin{figure}
$$
\xymatrix@R=4pt@C=4pt{
*={\bullet} \ar@{-}[rd] & {} & {} & {}  & *={\bullet} \ar@{-}[ld] \\
 & *={\bullet} \ar@{-}[rd] & {} & *={\bullet} \ar@{-}[ld]  & \\
{} & {} & *={\bullet}\ar@{-}[d] & {} & {} \\
{} & {} & *={\bullet} \ar@{-}[d] & {} & {} \\
{} & {} & *={\bullet} & {} & {}
}
\qquad
\xymatrix@R=5pt@C=0pt{
*={\bullet} \ar@{-}[rd] & {} & {} & {} & *={\bullet} \ar@{-}[ld]\\
{} & *={\bullet} \ar@{-}[rd] \ar@{-}[rr] & {} & *={\bullet} \ar@{-}[ld] & {}\\
{} & {} & *={\bullet} \ar@{-}[d] & {} &  {}\\
{} & {} & *={\bullet} & {} &  {}
}
\qquad
\xymatrix@R=8pt@C=1pt{
*={\bullet} \ar@{-}[rr] \ar@{-}[rd] & {} & *={\bullet} \ar@{-}[ld] \ar@{-}[rr] \ar@{-}[rd] & {} & *={\bullet}  \ar@{-}[ld]\\
{} & *={\bullet} \ar@{-}[rr] \ar@{-}[rd] & {} & *={\bullet} \ar@{-}[ld] & {}  \\
{} & {} & *={\bullet} & {} & {}
}
$$
\caption{The graphs $\Y_{3,3,3}$, $\triangle_{2,2,2}$ and $\triangle^*_{1,1,1}$, left to right}\label{Deltastar}
\end{figure}
Taking $\Sigma = \Y_{n+2,n'+2,n''+2}$ in \Tref{nice}, we obtain a finite-index embedding of simply laced Coxeter groups,
\begin{equation}\label{WW}
W(\triangle^*_{n,n',n''}) \hra W(\triangle_{n+1,n'+1,n''+1}),
\end{equation}
of index $n+n'+n''+4$. In our terms this is the embedding
$$C(\triangle_{n+1,n'+1,n''+1}) \hra C(\Y_{n+2,n'+2,n''+2}).$$
\end{exmpl}

\begin{rem}\label{n=0}
In \Eref{ex69} we assume $n,n',n''\geq 1$. With respect to the diagrams in \Fref{Deltastar}, taking (for example) $n = 0$ amounts to removing one vertex from each of $\Y_{n+2,n'+2,n''+2}$, $\triangle_{n+1,n'+1,n''+1}$, and $\triangle^*_{n,n',n''}$.
\end{rem}

By \Rref{YY+}, all the cases in which \Tref{nice} embeds a {\it{simply-laced}}\, Coxeter group in the infinite group $W(\Line{\Sigma})$ are given by \Eq{WW} for $n,n',n'' \geq 0$ (interpreted by \Rref{n=0}). In this notation, the example given in the Introduction is $W(\triangle^*_{0,1,1}) \hra W(\triangle_{1,2,2})$. \Eref{DY3}.\eq{DY3.emb} is $W(\triangle^*_{0,0,0}) \hra W(\triangle_{1,1,1})$.

\begin{rem}
To illustrate the connection between the various objects, note that the projection $\Cover[\Sigma] \ra W(\Line{\Line{\Sigma}})$ positions the stabilizer $C(\Sigma)_o$ both as a subgroup of $C(\Sigma)$ and as a cover of $C(\Line{\Sigma})$, by the maps:
$$\xymatrix{C(\Sigma)_o \ar@{^(->}[r] & C(\Sigma) \ar@{=}[r] & W(\Line{\Sigma})
\\
\Cover \ar@{->}[u]|\cong^{\Phi_o} \ar@{->>}[r] & C(\Line{\Sigma}) \ar@{=}[r]  & W(\Line{\Line{\Sigma}})
}$$
\end{rem}

\medskip
\section{Arbitrary graphs and relative conjugation extensions}\label{sec:ex}

In \Sref{sec:6} we identified the stabilizer $\GG_0(\Sigma) \isom C(\Sigma)_o$ as a Coxeter group, when~$\Sigma$ is a tree. We also know that the stabilizer $C(\Gamma)_o$, for an arbitrary graph $\Gamma$, is isomorphic to $\GG_0(\Sigma,\Gamma)$. In this section we show that if $\pi_1(\Gamma)$ is of rank $1$, then $\GG_0(\Sigma,\Gamma)$ admits a relative presentation over a Coxeter group with one additional generator and only conjugation relations.

In order to get a better understanding of $\GG_0(\Sigma,\Gamma)$, we need to know that the canonical map $\GG_0(\Sigma) \ra \GG_0(\Sigma,\Gamma)$ is an embedding.
As before, let $\Gamma$ be a finite simple graph, and $\Sigma \sub \Gamma$ a spanning subtree. Let $\Sigma \sub \Gamma' \sub \Gamma$ be an intermediate spanning subgraph.
\begin{prop}
The map $\iota \co \GG_0(\Sigma,\Gamma') \ra \GG_0(\Sigma,\Gamma)$,
defined by sending each generator $(\omega,x)$ or $x_a$ of $\GG_0(\Sigma,\Gamma')$ (as described in \Pref{thepres}) to its respective counterpart in $\GG_0(\Sigma,\Gamma)$, is an injection.
\end{prop}
\begin{proof}
Fix a vertex $o$ in $\Gamma$. The diagram
$$\xymatrix@R=12pt@C=16pt{
\GG_0(\Sigma,\Gamma) \ar@{->}[r]^{\Psi_o} & C(\Gamma)_o \ar@{^(->}[r] & C(\Gamma) \\
\GG_0(\Sigma,\Gamma') \ar@{->}[u]^{\iota} \ar@{->}[r]^{\Psi'_o} & C(\Gamma')_o \ar@{->}[u] \ar@{^(->}[r] & C(\Gamma') \ar@{^(->}[u] \\
}$$
commutes by the formula for $\Psi_o$ and $\Psi_o'$ in \Rref{whatisPsi}. The map $C(\Gamma') \ra C(\Gamma)$ is an embedding by \Pref{parabolic}, and
consequently  $C(\Gamma')_o \ra C(\Gamma)_o$ is an embedding. Since the $\Psi_o$ are isomorphisms, the map $\iota$ is an injection as well.
\end{proof}

In particular, taking $\Gamma' = \Sigma$, we get:
\begin{cor}
The map $\iota \co \GG_0(\Sigma) = \GG_0(\Sigma,\Sigma) \ra \GG_0(\Sigma,\Gamma)$ is an injection.
\end{cor}

The advantage here is that the isomorphism $\Theta \co \GG_0(\Sigma) \ra \Cover$, defined in \Pref{Thetadef}, provides us with a concise presentation, which can be used to give a better presentation for $\GG_0(\Sigma,\Gamma)$. Let us phrase this as an algorithm whose output is a presentation of $\GG_0(\Sigma,\Gamma) \isom C(\Gamma)_o$.

\begin{rem}[Simplified presentation for $C(\Gamma)_o$]\label{proc1}
Let $\Gamma$ be a simple graph, and $\Sigma$ a spanning subtree. A simplified presentation of $\GG_0(\Sigma,\Gamma)$ can be obtained from the one given in \Pref{thepres}, by applying $\Theta$ to the generators $(\omega,x)$ for $x \in E(\Sigma)$; moreover the relations involving pairs of generators $x,y$ from the spanning subtree become redundant in the presence of $\Cover$.
After this procedure, $\GG_0(\Sigma,\Gamma)$ is presented as generated by the subgroup $\Cover$ and two extra families of generators,  $(\omega,x)$ for $\omega \not \in x$, and $x_a$ for $a \in x$, ranging only over the $x \in E(\Gamma) - E(\Sigma)$.
\end{rem}

\begin{rem}\label{proc2}
Especially powerful in the procedure of \Rref{proc1} is the relation~\eq{54.2c} from \Pref{thepres}: if $x \not \in E(\Sigma)$, then $(c,x) = (d,x)$ for any edge $y = \set{c,d} \in E(\Sigma)$ disjoint from $x$. Letting $\Sigma - x$ denote the forest obtained by deleting the vertices of $x$ from $\Sigma$, there is thus only one generator of the form $(\omega,x)$ for each connected component of $\Sigma - x$.
\end{rem}

To gain a better understanding of $\GG_0(\Sigma,\Gamma)$, we focus on the part generated by the generators $(\omega,x)$, leaving aside the~$x_a$.
\begin{defn}\label{Ndef}
Let $\Gamma$ be a connected simple graph with a spanning subtree $\Sigma$. We define $N_0(\Sigma,\Gamma)$ to be the group generated by all the pairs $(\omega,x)$
for any edge~$x$ of~$\Gamma$ and any vertex~$\omega$ not on $x$;
subject to the relations:
\begin{enumerate}
\item $(\omega,x)^2 = 1$ \quad for every $x = \set{a,b} \in E(\Gamma)$ and $\omega \neq a,b$;
\item For every pair $x = \set{a,b}$ and $y = \set{c,d}$ of disjoint edges in $\Gamma$:
\begin{enumerate}
\item $[(\omega,x),(\omega,y)] = 1$ \quad for every $\omega \neq a,b,c,d$;
\item $(b,y) = (a,y)$ \quad if $x \in E(\Sigma)$;
\item $(d,x) = (c,x)$ \quad if $y \in E(\Sigma)$;
\end{enumerate}
\item For every pair $x = \set{a,b}$ and $y = \set{b,c}$ of (distinct) intersecting edges in~$\Gamma$:
\begin{enumerate}
\item $(\omega,x)(\omega,y)(\omega,x) = (\omega,y)(\omega,x)(\omega,y)$ \quad for every $\omega \neq a,b,c$;
\item $(c,x) = (a,y)$ \quad if $x,y \in E(\Sigma)$.
\end{enumerate}
\end{enumerate}
\end{defn}

Two things are clear from this definition. First, that $N_0(\Sigma,\Gamma)$ is a Coxeter group. And second, that there is a natural map $N_0(\Sigma,\Gamma) \ra \GG_0(\Sigma,\Gamma)$ sending each $(\omega,x)$ to its counterpart.
However, it should be noted that $N_0(\Sigma,\Gamma)$ may be missing relations induced by conjugation (as demonstrated in the proof of the relation $s_0ts_0=ts_0t$ in \Pref{cyctail}). Thus, in general, the map $N_0(\Sigma,\Gamma) \ra \GG_0(\Sigma,\Gamma)$ is not injective.
If $\Gamma = \Sigma$ then it is easy to see that $N_0(\Sigma,\Gamma) = \GG_0(\Sigma,\Gamma)$ (as already hinted in \Cref{relstree}). Moreover, $\GG_0(\Sigma)$ is a parabolic subgroup of $N_0(\Sigma,\Gamma)$, generated by the $(\omega,x)$ for $x \in E(\Sigma)$.

We shall call a relative presentation of the form
$$\sg{N,q \subjectto q u_\lam q^{-1}=v_\lam, \quad u_\lam,v_\lam \in N},$$
a {\bf{relative conjugation presentation}} over $N$, and the group so presented a {\bf{relative conjugation extension}} of $N$. This terminology refers only to the form of the presentation: it does not assert that the natural map from $N$ is injective, or that the prescribed pairs extend to an isomorphism between subgroups of $N$. If the conjugation induces an isomorphism from the subgroup generated by the $u_\lam$ to the subgroup generated by the $v_\lam$, this becomes an HNN extension.

\smallskip
\subsection{The stabilizer for unicyclic graphs}

Let $\Gamma$ be a connected simple graph with a unique cycle. Taking $\Sigma$ to be a spanning subtree, there is a unique edge $p \in \Gamma$ such that $\Gamma = \Sigma \cup \set{p}$. We will give a relative conjugation presentation of $\GG_0(\Sigma,\Gamma)$ over a Coxeter group.


Let $a,b$ be the vertices on the edge $p$. From the definition it is clear that  $\GG_0(\Sigma, \Gamma)$ is generated by the image of the map $N_0(\Sigma,\Gamma)$ and the generator $p_a$. A thorough inspection of \Pref{thepres} leads to the following:
\begin{prop}\label{PresCase1}
Assume $\Gamma = \Sigma \cup \set{p}$. The group $\GG_0(\Sigma,\Gamma)$ is generated by $N_0(\Sigma,\Gamma)$, and the generator $p_a$, subject to the following relations:
\begin{enumerate}
\item $p_a(b,y)p_a^{-1} = (a,y)$ \qquad for every edge $y = \set{c,d}$ disjoint from $p$;
\item $p_a(c,p)p_a^{-1} = (a,y)$ \qquad for every edge $y = \set{b,c}$;
\item $p_a(b,y)p_a^{-1} = (c,p)$ \qquad for every edge $y = \set{a,c}$.
\end{enumerate}
\end{prop}

\begin{cor}\label{caseuni}
Assume that $\Gamma=\Sigma\cup\set{t}$, where $t=\set{a,b}$. Then $\GG_0(\Sigma,\Gamma)$ admits the following relative presentation over $N_0(\Sigma,\Gamma)$:
$$
\begin{aligned}
\GG_0(\Sigma,\Gamma) =\big\langle N_0(\Sigma,\Gamma),t_a\ \big|\;&  t_a(b,y)t_a^{-1}=(a,y)  &&\text{for $y\in E(\Gamma)$ disjoint from $t$},\\
&  t_a(c,t)t_a^{-1}=(a,\set{b,c})  &&\text{for $\set{b,c}\in E(\Sigma)$},\\
&  t_a(b,\set{a,c})t_a^{-1}=(c,t)  &&\text{for $\set{a,c}\in E(\Sigma)$} \big\rangle .
\end{aligned}$$
Thus $\GG_0(\Sigma,\Gamma)$ is a relative conjugation extension of $N_0(\Sigma,\Gamma)$ in the terminology above: it is obtained by adjoining one generator and imposing the indicated conjugation relations.
\end{cor}
Notice that the natural map $N_0(\Sigma,\Gamma) \ra \GG_0(\Sigma,\Gamma)$ is not injective in general, see \Eref{badN0bad}.

The rest of this subsection is devoted to a family of relatively simple examples.

\smallskip
\subsection{Graphs with a spanning path}\label{ss:71}

Preparing for the following subsections, consider the path graph $\Sigma = \path[n]$. Denote the vertices by $1,\dots,n$, and let $u_i = \set{i,i+1}$ be the edges, $i = 1,\dots,n-1$. The vertices of the line graph $\Line{\path[n]}$ are $u_1,\dots,u_{n-1}$. The vertices in $\covg{\path[n]} = \Line{\Line{\path[n]}}$ are the edges $\set{u_i,u_{i+1}}$ of $\Line{\path[n]}$, which we denote as $s_j = \set{u_j,u_{j+1}} = [j\,j+1\,j+2]$,\ $j = 1,\dots,n-2$. Note that $\Cover[{\path[n]}] \isom S_{n-1}$, with the Coxeter generators $s_1,\dots,s_{n-2}$. By definition, the isomorphism $\Theta \co \GG_0(\Sigma) \ra \Cover \isom S_{n-1}$ is defined for $\omega \neq i,i+1$ by
$$\Theta \co (\omega,u_i) \mapsto \begin{cases}
s_{i-1}& \omega<i \\ 
s_{i}& i+1<\omega\end{cases}.$$

Note that if $\path[n]$ is a spanning subtree of any graph $\Gamma$, and $x$ is a new edge of $\Gamma$, then by \Rref{proc2} there are at most three distinct generators of the form $(\omega,x)$.

\smallskip
\subsection{The cycle}

Let us now consider the cycle graph $\Gamma = \cyc$, for which $C(\Gamma)$ is virtually abelian. First we build on the Coxeter representation.
\begin{exmpl}\label{cyccase}
Let $\Gamma = \cyc$ be the cycle graph for $n> 3$. Letting $S_n$ act naturally on~$\Z^n$, denote by~$(\Z^n)_0$ the submodule of zero-sum vectors. The Coxeter group $C(\cyc)$ is the semidirect product $S_n \ltimes (\Z^n)_0$, as discussed in \Eref{cn}.
Let $C(\cyc)_o = S_{n-1} \ltimes (\Z^n)_0$ denote the stabilizer of the vertex $n$.
We have a short exact sequence $$0 \lra (\Z^{n-1})_0 \lra (\Z^n)_0 \lra \Z \lra 0$$
of $S_{n-1}$-modules, 
where the projection is onto the $n$th component.
Therefore, we have a short exact sequence of groups,
\begin{equation}\label{Ccyc0}
1 \lra C(\cyc[n-1]) \lra C(\cyc)_o \lra \Z \lra 1,
\end{equation}
embedding $C(\cyc[n-1])$ as a normal subgroup of $C(\cyc)_o$.
\end{exmpl}
A presentation, which we are about to exhibit, will describe the action of the cyclic quotient on $C(\cyc[n-1])$ which is implicit in \eq{Ccyc0}. First we treat the case $n = 3$.

\begin{exmpl}
Let $\Gamma$ be the graph $\triangle$, as given in \Fref{basic}, with a spanning subtree $\Sigma = \set{x,y}$. In \Eref{DY2} we showed that $C(\triangle) = S_3 \ltimes (\Z^3)_0$ where the normal subgroup is $\sg{\alpha,\beta,\gamma}$. Let $o$ denote the vertex $o = y \cap z$. Then $C(\triangle)_o = \sg{x, \alpha,\beta,\gamma}$ where $\alpha\beta\gamma=1$. (Compare to \Eref{DY5}.)
Notice that $x\alpha x^{-1} = \alpha^{-1}$. We have a short exact sequence
\begin{equation}\label{Ccyc0;n=3}
1 \lra \sg{x,\alpha} \lra C(\triangle)_o \lra \Z \lra 1,
\end{equation}
where the map to $\Z$ is defined by $x \mapsto 0$, $\alpha \mapsto 0$, $\beta \mapsto 1$ and $\gamma \mapsto -1$. The kernel is $\sg{x,\alpha} \isom D_{\infty}$, the Coxeter group of type $\tilde{A}_1$.
\end{exmpl}

We keep the notation from \Ssref{ss:71}. Realize $\cyc[n]$, where $n>3$, as the graph obtained by adding the edge $p = \set{1,n}$ to the path $\path[n]$. Applying the procedure of \Rref{proc1}, we obtain the following presentation.
\begin{prop}\label{thepres.which}
The group $\GG_0(\path[n],\cyc[n])$ is generated by $\Cover[{\path[n]}] = \sg{s_1,\dots,s_{n-2}}$ and the generators of the form $(\omega,p)$ for $\omega \not \in \set{1,n}$ and $p_1,p_n$,
subject to the following relations:
\begin{enumerate}
\item
\begin{enumerate}
\item $(\omega,p)^2 = 1$ \quad for every\ $\omega \neq 1,n$;
\item\label{TC.1b} $p_1p_n = 1$;
\end{enumerate}
\item For every\ $i = 1,\dots,n-3$: 
\begin{enumerate}
\item $[s_i,(\omega,p)] = 1$ \quad for every\ $1 < \omega < i+1$;
\item $[s_{i+1},(\omega,p)] = 1$ \quad for every\ $i+2 < \omega < n$;
\item $s_{i+1} = p_1^{-1} s_{i} p_1$;
\item\label{meme} $(i+1,p) = (i,p)$;
\end{enumerate}
\item
\begin{enumerate}
\item $s_1(\omega,p)s_1 = (\omega,p)s_1(\omega,p)$ \quad for every $\omega \neq 1,2,n$;
\item $p_n^{-1}s_1p_n = (2,p)$; 
\end{enumerate}
\item
\begin{enumerate}
\item $s_{n-2}(\omega,p)s_{n-2} = (\omega,p)s_{n-2}(\omega,p)$ \quad for every $\omega \neq 1,n-1,n$;
\item $p_1^{-1}s_{n-2}p_1 = (n-1,p)$. 
\end{enumerate}
\end{enumerate}
\end{prop}

Taking advantage of \Rref{proc2}, we notice that all the generators $(\omega,p)$ coincide, so we (purposefully) denote $(\omega,p) = s_0$. We can thus translate the previous presentation, still assuming $n>3$, to:
\begin{cor}\label{cor76}
The group $\GG_0(\path[n],\cyc[n])$ is generated by $\Cover = \sg{s_1,\dots,s_{n-2}}$ and the generators $s_0$ and $p_1$,
subject to the following relations:
\begin{enumerate}
\item
\begin{enumerate}
\item $s_0^2 = 1$;
\item $[s_{i},s_0] = 1$ \quad for $i = 2,\dots,n-3$:
\item $s_1s_0s_1 = s_0s_1s_0$;
\item $s_{n-2}s_0s_{n-2} = s_0s_{n-2}s_0$;
\end{enumerate}
\item
\begin{enumerate}
\item $p_1s_ip_1^{-1} = s_{i-1}$ \quad for $i = 1,\dots,n-2$;
\item $p_1s_0p_1^{-1} = s_{n-2}$.
\end{enumerate}
\end{enumerate}
\end{cor}
The first part of \Cref{cor76} shows that $H = \sg{s_0,s_1,\dots,s_{n-2}}$ is the Coxeter group of type $\tilde{A}_{n-2}$, isomorphic to $C(\cyc[n-1])$. The second part shows that~$H$ is a normal subgroup of $\GG_0(\path[n],\cyc[n]) \isom C(\cyc[n])_o$, on which $p_1$ acts by rotating the generators. In particular $p_1^{n-1}$ is a central element, in fact generating the center. Letting $\rho \co C(\cyc[n-1]) \ra C(\cyc[n-1])$ denote this rotation map of order $n-1$, we conclude that
\begin{cor}
The stabilizer $C(\cyc[n])_o$, where $n>3$, is an HNN extension:
$$C(\cyc[n])_o = \HNN(C(\cyc[n-1]);\rho).$$
\end{cor}
This gives a more explicit form to the short exact sequence~\eq{Ccyc0}.

\smallskip
\subsection{A cycle with a tail}\label{ss:tail}

Now consider the union of $\path[n]$ and the edge $p = \set{1,k}$, where $3 \leq k \leq n-1$. This is the union of a cycle $\cyc[k]$ with a path $\path[n-k+1]$, glued in a common vertex at the end of the path. We denote this (unambiguous) gluing by $\cycpath{k}{n-k+1}$. Clearly, $\path[n]$ may be taken as a spanning subtree of $\cycpath{k}{n-k+1}$.

\begin{figure}
$$\xymatrix@C=24pt@R=12pt{
1 \ar@{-}[r]_{u_1} \ar@/^24pt/@{-}[rrrr]_p & 2 \ar@{-}[r]_{u_2} & 3 \ar@{.}[r] & k\!-\!1 \ar@{-}[r]_(0.6){u_{k-1}} & k \ar@{-}[r]_(0.4){u_k} & k\!+\!1 \ar@{.}[r] & n-1 \ar@{-}[r]_(0.6){u_{n-1}} & n \\
}
$$
\caption{The graph $\cycpath{k}{n-k+1}$ studied in \Ssref{ss:tail}}\label{tailfig}
\end{figure}

Thus, take $\Gamma = \cycpath{k}{n-k+1}$, as depicted in \Fref{tailfig}.  By \Rref{proc2}, $(2,p) = \cdots = (k-1,p)$, and $(k+1,p) = \cdots = (n,p)$. We denote the former by $s_0$ and the latter by $t$. Also note that analogously to \Pref{thepres.which}.\eq{TC.1b}, here $p_1p_k = 1$. Let us now apply the procedure of \Rref{proc1} as above.

\begin{prop}\label{cyctail} 
The group $\GG_0(\path[n], \cycpath{k}{n-k+1})$ is generated by $\Cover[{\path[n]}] = \sg{s_1,\dots,s_{n-2}} \isom S_{n-1}$ and the generators $s_0$, $t$ and~$p_1$, subject to the following relations:
\begin{enumerate}
\item\label{XT1} $s_0^2 = t^2 = 1$;
\item \label{XT2}
\begin{enumerate}
\item\label{XT2a} $s_1s_0s_1 = s_0s_1s_0$ \qquad if $4 \leq k$;
\item\label{XT2b} $[s_i,s_0] = 1$ \qquad ($i = 2,...,k-3$);
\item\label{XT2c} $s_{k-2}s_0s_{k-2} = s_0s_{k-2}s_0$ \qquad if $4 \leq k$;
\item\label{XT2d} $s_{k-1}s_0s_{k-1} = s_0s_{k-1}s_0$;
\item\label{XT2e} $[s_i,s_0] = 1$ \qquad ($i = k,...,n-2$);
\end{enumerate}
\item \label{XT3}
\begin{enumerate}
\item\label{XT3a} $s_1ts_1 = ts_1t$;
\item\label{XT3b} $[s_i,t] = 1$ \qquad ($i = 2,...,k-2$);
\item\label{XT3c} $s_{k-1}ts_{k-1} = ts_{k-1}t$;
\item\label{XT3d} $s_{k}ts_{k} = ts_{k}t$ \qquad if $k \leq n-2$;
\item\label{XT3e} $[s_i,t] = 1$ \qquad ($i = k+1,...,n-2$);
\end{enumerate}
\item\label{XT5} $s_0 t s_0 = t s_0 t$;
\item \label{XT4}
\begin{enumerate}
\item\label{XT4a} $p_1 s_i p_1^{-1} = s_{i-1}$ \qquad ($i = 1,\dots,k-2$);
\item\label{XT4b} $p_1 s_0 p_1^{-1} = s_{k-2}$;
\item\label{XT4c} $p_1 t p_1^{-1}=  s_{k-1}$;
\item\label{XT4d} $p_1 s_{i} p_1^{-1} = s_{i}$ \qquad ($i = k,\dots,n-2$).
\end{enumerate}
\end{enumerate}
\end{prop}
\begin{proof}
This is \Rref{proc1} applied to the presentation of \Pref{thepres} for $\Gamma = \cycpath{k}{n-k+1}$. After some reordering, this gives \eq{XT1},\eq{XT2}, \eq{XT3} and \eq{XT4}.  The conditional relations \eq{XT2a} and \eq{XT2c} are obtained from relations \eq{54.3a} in \Pref{thepres} only when the necessary~$\omega$ exists, which is the case only when $k \geq 4$.
The relation~\eq{XT5} in this list, $s_0 t s_0 = t s_0 t$, is {\emph{not}} a direct outcome of the procedure; however it does follow by conjugating $s_{k-2}s_{k-1}s_{k-2}=s_{k-1}s_{k-2}s_{k-1}$ by $p_1^{-1}$, using relations \eq{XT4b} and~\eq{XT4c}.
\end{proof}

In the group
$\GG_0(\path[n],\cycpath{k}{n-k+1}) = \sg{p_1,t,s_0,s_1\dots,s_{n-2}}$ defined in \Pref{cyctail}, let $$P = \sg{s_0,s_1,\dots,s_{k-2},s_k,\dots,s_{n-2}}.$$
Let $\varphi \in \Aut(P)$ be defined on $\sg{s_0,\dots,s_{k-2}}$ by the rotation $s_i \mapsto s_{i-1 \pmod {k-1}}$, and on $s_k,\dots,s_{n-2}$ as the identity. Extend this automorphism to an isomorphism $\varphi \co \sg{P,t} \ra \sg{P,s_{k-1}}$ by sending $t \mapsto s_{k-1}$.

\begin{rem}\label{oops}
\Dref{CT} defines the group $C(\Gamma)$ when $\Gamma$ is a simple graph. The definition can be extended to any graph: take the same relations, where if $u,u'$ have the same vertices then the respective pair of generators satisfy no relation. For example, if $\Pi$ is the graph depicted in \Fref{Gammastar} for $k = 3$ and $n = 4$, then $xyx = yxy$ for any two generators of $C(\Pi)$ other than $s_0,s_1$.
\end{rem}

\begin{cor}\label{O--}
The group $\GG_0(\path[n],\cycpath{k}{n-k+1})$ is an HNN extension of a Coxeter group via an isomorphism of parabolic subgroups. Explicitly,
\begin{enumerate}[label=(\alph*)]
\item\label{MM1} $P = \sg{s_0,\dots,s_{k-2}, s_k,\dots,s_{n-2}}$ is the Coxeter group of type $\tilde{A}_{k-2} \times A_{n-k-1}$;
\item\label{MM2} $\sg{P,s_{k-1}} \isom \sg{P,t}$ are isomorphic to $C(\cycpath{k-1}{n-k+1})$; 
\item\label{MM3} $\sg{P,s_{k-1},t}$ is the Coxeter group $C(\Pi)$, where $\Pi$ is the graph given in \Fref{Gammastar} (\Rref{oops} applies if $k = 3$).
\item and finally,
$$\GG_0(\path[n],\cycpath{k}{n-k+1}) = \HNN(\sg{P,s_{k-1},t}; \, \sg{P,t},\, \sg{P,s_{k-1}}, \, \varphi).$$
\end{enumerate}
\end{cor}
\begin{proof}
Starting from the final statement, we observe that relation~\eq{XT4} defines the conjugation action of $p_1$ from the subgroup $\sg{P,t}$, into the subgroup $\sg{P,s_{k-1}}$, which restricts to $\varphi$ on $P$. This gives $\GG_0(\path[n],\cycpath{k}{n-k+1})$ the structure of an HNN extension $\HNN(G; A,B,\varphi)$, where $G = \sg{P,s_{k-1},t}$ and, crucially, the defining relations of $G$ are \eq{XT1},\eq{XT2},\eq{XT3},\eq{XT5}. Inspecting all these relations shows that $G = C(\Pi)$ where $\Pi$ is the graph given in \Fref{Gammastar}.

Once we know $G$ is a Coxeter group, the parabolic subgroups are of the form $C(\cdot)$, where the graph is obtained by deleting edges, which proves \ref{MM2} and \ref{MM3}. It also shows that $P = C(\cyc[k-1] \cup \path[n-k])$, a disjoint union, which proves \ref{MM1}. Here~$\cyc[2]$ denotes a pair of parallel edges on two vertices, for which $C(\cyc[2]) = D_{\infty}$ is the Coxeter group of type $\tilde{A}_1$.
\end{proof}

\begin{figure}
$$\xymatrix@C=12pt@R=12pt{
 \circ \ar@/_6pt/@{-}[dd]_{s_1} \ar@/^6pt/@{-}[dd]^{s_0} \ar@{-}[rrd]^{s_2} & {} & {}& {} & {} & {} & {}& {}\\
& {}{} & \circ \ar@{-}[r]^{s_3} & \circ \ar@{.}[r] & \circ \ar@{-}[r]^{s_{n-2}} & \circ \\
 \circ \ar@{-}[rru]_{t} & {}& {} & {} & {} & {} & {} & {}
}
\quad
\xymatrix@C=12pt@R=12pt{
\circ \ar@{-}[r]^{s_{k-2}} \ar@/_24pt/@{.}[dd] & \circ \ar@{-}[rd]^{s_{k-1}} \ar@{-}[dd]_{s_0} & {} & {} & {} & {} \\
{} & {} & \circ \ar@{-}[r]^{s_{k}} & \circ \ar@{.}[r] & \circ \ar@{-}[r]^{s_{n-2}} & \circ \\
\circ \ar@{-}[r]^{s_1} & \circ \ar@{-}[ru]_t & {} & {} & {} & {} \\
}
$$
\caption{Graphs for the stabilizer $C(\cycpath{k}{n-k+1})_o$: the cases $k = 3$ (left), $k \geq 4$ (right).
}\label{Gammastar}
\end{figure}
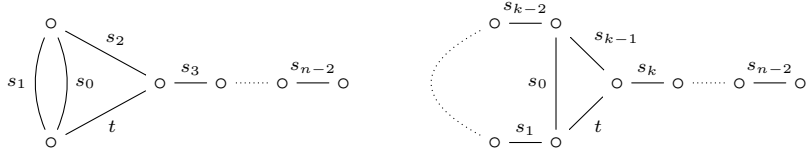

After presenting $\GG_0(\path[n],\cycpath{k}{n-k+1})$ as an HNN extension, let us contrast this presentation with \Cref{caseuni}.
\begin{exmpl}\label{badN0bad}
Consider the simplest case of a cycle with tail, namely $\Gamma = \cycpath{3}{2}$ with $\Sigma = \path[4]$ as above. This is the case $k = 3$ and $n = 4$ above. Working through the definition, $N_0(\Sigma,\Gamma) = \sg{s_0,s_1,s_2,t}$, where $s_i^2 = t^2 = 1$, and $(s_0s_2)^3 = (s_1s_2)^3 = (s_1t)^3 = (s_2t)^3 = 1$. The relation $(s_0t)^3 = 1$ is not assumed, and since this is a Coxeter presentation, $s_0t$ has infinite order in $N_0(\Sigma,\Gamma)$. However this relation does hold in $\GG_0(\Sigma,\Gamma)$, by letting $p_1^{-1}$ act on $(s_1s_2)^3 = 1$. Hence $N_0 \ra \GG_0$ is not an embedding.
\end{exmpl}

\section{The structure of $\GG_0(\Sigma,\Gamma)$}\label{sec8}

In this final section we describe the stabilizer $\GG_0(\Sigma,\Gamma)$ by an iterated relative conjugation presentation, with $\card{\Gamma-\Sigma} = \rank \pi_1(\Gamma)$ steps. The presentation begins with the Coxeter group $N_0(\Sigma,\Gamma)$ defined in \Dref{Ndef}.

Let $\Gamma$ be a connected simple graph with a spanning subtree $\Sigma$. We fix an ordering of the new edges: $E(\Gamma) = E(\Sigma) \cup \set{x^{(1)},\dots,x^{(t)}}$. For each $i$ fix the vertices $a_i,b_i$ of the edge $x^{(i)}$.

For $j = 1,\dots,t$, let $N_j(\Sigma,\Gamma)$ be the group generated by $N_0(\Sigma,\Gamma)$ and the generators $x^{(i)}_{a_i},\, x^{(i)}_{b_i}$ for $i \leq j$, subject to the relations given in the presentation of $\GG_0(\Sigma,\Gamma)$ (\Pref{thepres}) in which those generators participate. Identification of generators defines a sequence of maps $N_0 \ra N_1 \ra \cdots \ra N_t = \GG_0(\Sigma,\Gamma)$, which a priori may not be injective.

\begin{prop}\label{HNN4}
With the notation above, $N_j$ admits a relative conjugation presentation over $N_{j-1}$, with one additional generator.
\end{prop}
\begin{proof}
Let $z_j = x^{(j)}_{a_j}$ be the new generator, noting that $x^{(j)}_{b_j} = z_j^{-1}$. Then $N_j$ is obtained from $N_{j-1}$ by adding the generator $z_j$ and a (finite) set of relations of the form $z_j w z_j^{-1} = w'$ for $w,w' \in N_{j-1}$.
\end{proof}

We thus proved the following generalization of \Cref{caseuni}:
\begin{thm}\label{thegen}
Let $\Gamma$ be a connected simple graph with a spanning subtree $\Sigma$. The stabilizer $\GG_0(\Sigma,\Gamma)$ admits an iterated relative conjugation presentation beginning with the Coxeter
group $N_0(\Sigma,\Gamma)$ and adjoining $t = \rank \pi_1(\Gamma)$ generators, one at each step.
\end{thm}

As in \Cref{caseuni}, for $j=1$ the elements occurring on the two sides of the conjugation relations are Coxeter generators of $N_0(\Sigma,\Gamma)$, and the two collections generate parabolic subgroups. We do not assert that the resulting pairing extends to an isomorphism of these subgroups, or that the natural map $N_0(\Sigma,\Gamma) \ra N_1(\Sigma,\Gamma)$ is injective. For $j>1$, even the two sides of a conjugation relation need not both lie in $N_0(\Sigma,\Gamma)$: relation \ref{thepres}.\eq{54.3b} may let $z_2$ conjugate elements of the form $z_1wz_1^{-1}$ and $w'$ for $w,w' \in N_0(\Sigma,\Gamma)$, while $z_1 w z_1^{-1} \not \in N_0(\Sigma,\Gamma)$ in general.

Let $N$ be a group and, for $i=1,\dots,t$, let $\mathcal R_i \subseteq N\times N$ be a collection of prescribed pairs. The {\bf{simultaneous relative conjugation extension}} associated to these data is the group
$$
\sg{N,q_1,\dots,q_t \subjectto q_i u q_i^{-1}=v\quad ((u,v)\in\mathcal R_i,\ i=1,\dots,t)}.
$$

We say that a spanning subtree $\Sigma$ of a simple graph $\Gamma$ is {\bf{separating}} if $\Gamma-\Sigma$ is a collection of disjoint edges. A connected simple graph with a separating subtree is {\bf{thin}}. When this is the case, in every instance of the relation \ref{thepres}.\eq{54.3b}, either $x_b = 1$ or $y_b = 1$ by \ref{thepres}.\eq{54.1c}, so all the conjugation relations are to elements of $N_0(\Sigma,\Gamma)$.

This proves:
\begin{thm}\label{thethin}
Let $\Gamma$ be a thin connected simple graph with a separating spanning subtree $\Sigma$. The stabilizer $\GG_0(\Sigma,\Gamma)$ admits a simultaneous relative conjugation presentation over the Coxeter
group $N_0(\Sigma,\Gamma)$, using $t = \rank \pi_1(\Gamma)$ additional generators. For each additional generator, the two collections of Coxeter generators occurring in its conjugation relations generate parabolic subgroups of $N_0(\Sigma,\Gamma)$. 
\end{thm}

\end{document}